\documentclass[leqno,12pt]{amsart}
\usepackage{amssymb}
\usepackage{amsmath}
\usepackage{enumerate}
\usepackage{amsfonts}
\usepackage{hyperref}
\usepackage{mathrsfs}
\usepackage{cancel}
\usepackage{tikz}
\usepackage{graphicx}

\usepackage[headheight=18pt, top=17mm, bottom=17mm, left=20mm, right=20mm]{geometry}

\hypersetup{
    colorlinks=true,
    linkcolor= blue,
    citecolor =cyan,
    urlcolor = teal,
}

\newtheorem{theorem}{Theorem}[section]
\newtheorem{corollary}[theorem]{Corollary}
\newtheorem{lemma}[theorem]{Lemma}
\newtheorem{proposition}[theorem]{Proposition}

\newtheorem{definition}[theorem]{Definition}

\numberwithin{equation}{section}

\begin{document}

\title[Product Hardy spaces
associated with Schr\"odinger operators]{Fefferman--Chang-type atomic decompositions for
product Hardy spaces associated with
Schr\"odinger operators}

\author[Agnieszka Hejna-Łyżwa]{Agnieszka Hejna-Łyżwa}

\address{Agnieszka Hejna-Łyżwa, Uniwersytet Wroc\l awski,
Instytut Matematyczny,
Pl. Grunwaldzki 2,
50-384 Wroc\l aw,
Poland}
\email{agnieszka.hejna@math.uni.wroc.pl}

\begin{abstract}
We establish a geometric (Fefferman--Chang-type) atomic characterization of the
two-parameter Hardy space associated with a pair of
Schr\"odinger operators $L_i=-\Delta+V_i$ on
$\mathbb R^{n_i}$, $i=1,2$.
The space is defined by the product conical square function.
We assume comparability of neighbouring covering cubes,
a local bound on the potentials, and decay of the
semigroup mass at large scales.

We prove that the Hardy norm is equivalent to a sum
of classical local product Hardy norms at scales
determined by the adapted coverings.
This localization formula yields an atomic decomposition
whose atoms satisfy size and partial cancellation
conditions, with cancellation required only below
the corresponding local scales.
Both the atoms and their rectangular components
remain supported inside fixed enlargements of the
covering rectangles.

The proof combines one-parameter Schr\"odinger Hardy
space characterizations with standard vector-valued
arguments and a localized product atomic decomposition
with uniform control of scales and supports.
The assumptions cover the model case of reverse
H\"older potentials $V\in RH_q$, $q>n/2$, as well
as potentials outside every $RH_q$ class, $q>1$,
including half-space indicators and exponential
potentials.
\end{abstract}



\maketitle

\section{Introduction}
The purpose of this paper is to obtain a geometric
atomic characterization of the two-parameter Hardy
space associated with a pair of Schr\"odinger operators
\[
    L_i=-\Delta_{x_i}+V_i,
    \qquad i=1,2,
\]
on $\mathbb R^{n_1}\times\mathbb R^{n_2}$.
The space is defined by the product conical square
function associated with the semigroups
$e^{-t_1L_1}$ and $e^{-t_2L_2}$, whose parameters
vary independently.
We seek atoms described by support, size, and
partial cancellation conditions, with local scales
determined by coverings adapted to the potentials.

The geometric structure of product atoms originates
in the work of Chang and Fefferman
\cite{ChangFefferman1980,ChangFefferman1982}.
Briefly, such an atom is supported in an open set
$\Omega$ of finite measure and admits a decomposition
\[
    a=\sum_{R\in\mathcal M(\Omega)}a_R,
    \qquad
    \operatorname{supp}a_R\subset cR,
    \qquad
    \|a\|_2^2+\sum_{R\in\mathcal M(\Omega)}\|a_R\|_2^2
    \lesssim|\Omega|^{-1},
\]
where $\mathcal M(\Omega)$ denotes the maximal
dyadic rectangles contained in $\Omega$ and
$c>1$ is fixed.
Each component satisfies cancellation in both variables:
\[
    \int_{\mathbb R^{n_1}}a_R(x_1,x_2)\,dx_1=0,
    \qquad
    \int_{\mathbb R^{n_2}}a_R(x_1,x_2)\,dx_2=0,
\]
for almost every value of the remaining variable.
Our atoms retain this open-set and rectangular
structure, but cancellation is required only
when the corresponding side length of $R$ is
smaller than the local scale determined by
the adapted covering.

In the one-parameter setting, the relation between
local scales and cancellation already appears
in Goldberg's local Hardy spaces \cite{Goldberg}.
For Schr\"odinger operators with reverse H\"older
potentials, Dziuba\'nski and Zienkiewicz
\cite{DZ_ReverseHolder} obtained atomic
characterizations governed by the critical radius introduced in the work of Fefferman~\cite{Fefferman}:
\[
    \rho_V(x)
    =
    \sup\left\{
        r>0:
        r^{2-n}\int_{B(x,r)}V(y)\,dy\leq1
    \right\}.
\]
Their atoms satisfy the usual support and size
conditions
\[
    \operatorname{supp}a\subset B(x,r),
    \qquad
    \|a\|_\infty\leq |B(x,r)|^{-1},
    \qquad
    r\lesssim\rho_V(x).
\]
Cancellation $\int a=0$ is required when
$r$ is sufficiently small relative to $\rho_V(x)$,
whereas atoms at the critical scale need not
have integral zero.
Their subsequent covering-based theory
\cite{DZ_Studia} extends this principle beyond
reverse H\"older potentials, with the side lengths
of adapted covering cubes replacing the critical
radius.

This suggests a geometric product analogue:
rectangular components should satisfy cancellation
in each variable whose side length is below
the corresponding local scale, with no cancellation
required in variables at or above that scale.
To the best of our knowledge, a characterization
by such Chang--Fefferman-type atoms, with independently
varying local scales in the two variables,
has not previously been established for the
class of Schr\"odinger operators considered here.
The covering-based theory of
\cite{DZ_Studia} provides the one-parameter
starting point for this construction.

Further results demonstrate the scope of this
local approach.
Czaja and Zienkiewicz established an atomic
characterization for one-dimensional Schr\"odinger
operators with nonnegative locally integrable
potentials \cite{CzajaZienkiewicz2008}.
Dziuba\'nski and Preisner obtained a Riesz transform
characterization under the covering and potential
conditions used in \cite{DZ_Studia}, and discussed
additional constructions of admissible examples
\cite{DziubanskiPreisner2011}.
Local atomic decompositions for multidimensional
operators were studied by Kania-Strojec, Plewa,
and Preisner \cite{KaniaStrojecPlewaPreisner2021},
including stability results for sums of operators
acting in separate variables.
Their setting involves a single semigroup parameter.
Here the two independent parameters require
product atoms and separate cancellation conditions
in the two spatial variables.

Hardy spaces associated with operators on product
domains have also been studied through operator-adapted
molecules and atoms.
Deng, Song, Tan, and Yan developed a theory of
Hardy and BMO spaces on product domains, including
a molecular decomposition obtained through tent
spaces \cite{DengSongTanYan2007}.
Atomic decompositions and square function and
maximal function characterizations for product
Schr\"odinger Hardy spaces were subsequently
established in
\cite{SongTanYan2011,SongTan2012}. Further characterizations involving combinations of
maximal functions and square functions in the two
variables were obtained by Liu and Zhao
\cite{LiuZhao2015} and Zhao, Liu, and Jiang
\cite{ZhaoLiuJiang2019}.
More recently, Deng and Guedjiba
\cite[Definition~3.5 and Theorem~4.1]{DengGuedjiba2026}
established atomic characterizations for product Hardy
spaces associated with nonnegative self-adjoint
operators satisfying Davies--Gaffney estimates.
Their rectangular components are represented through
powers of the underlying operators.
Here we obtain atoms defined by support, size, and
ordinary partial cancellation conditions, with the
dependence on the Schr\"odinger operators encoded
by the adapted coverings and their local scales.

The local product theory of Ding, Lu, and Zhu
\cite{DingLuZhu2019} provides the classical
two-parameter model for our construction.
Geometric atomic descriptions of local Hardy
spaces also occur in the multi-norm setting
studied in by the author, Nagel, and Ricci in
\cite{HejnaNagelRicci2025}.
In that setting, the geometry is specified by
families of dilations.
In the present paper, the local scales vary
with the spatial location through coverings
adapted to the Schr\"odinger potentials.

For each factor, we assume the covering
comparability condition \eqref{ass:F},
the local potential estimate \eqref{ass:K},
and the semigroup mass decay condition
\eqref{ass:D}.
Let $\mathcal Q_i$ be the corresponding coverings
and let $\{\phi_{Q_i}\}_{Q_i\in\mathcal Q_i}$
be subordinate smooth partitions of unity.
One of our main results is the localization formula
\[
    \|f\|_{H^1_{L_1,L_2}}
    \asymp
    \sum_{Q_1\in\mathcal Q_1}
    \sum_{Q_2\in\mathcal Q_2}
        \|\phi_{Q_1}\phi_{Q_2}f\|
            _{h^1_{\ell(Q_1),\ell(Q_2)}}.
\]
Here $h^1_{r_1,r_2}$ is the classical local
product Hardy space with independent local
scales $r_1$ and $r_2$.
Thus the Schr\"odinger Hardy norm is described
by a sum of classical local product norms.
The precise membership statement for $L^1$
functions is given in
Theorem~\ref{thm:main-localization}.

The atomic characterization is stated in
Theorem~\ref{thm:main-atomic-characterization}.
Each atom is associated with a pair
$Q_i\in\mathcal Q_i$ and is a local product atom
at scales $\eta\ell(Q_1),\eta\ell(Q_2)$,
where $\eta>0$ is fixed.
Its associated open set and all its rectangular
components are contained in
$Q_1^{**}\times Q_2^{**}$.
The components satisfy partial cancellation
whenever the corresponding side length is below
a fixed multiple of the relevant local scale.
In particular, the definition does not require
the components to be represented as images
of powers of $L_1$ and $L_2$. Figure~\ref{fig:geometric-atoms} illustrates the support
conditions and the four possible cancellation regimes
for the rectangular components of an atom.

\begin{figure}[htbp]
\centering
\resizebox{0.95\textwidth}{!}{%
\begin{tikzpicture}[
    font=\small,
    covering/.style={
        draw=black,
        line width=0.8pt
    },
    openset/.style={
        draw=black!60,
        fill=black!6,
        line width=0.7pt,
        rounded corners=4pt
    },
    component/.style={
        draw=black!55,
        line width=0.6pt
    },
    selected/.style={
        draw=blue!65!black,
        fill=blue!10,
        line width=0.9pt
    },
    enlarged/.style={
        draw=blue!65!black,
        densely dashed,
        line width=0.8pt
    },
    threshold/.style={
        draw=black!45,
        densely dashed,
        line width=0.6pt
    }
]

\begin{scope}
    \node at (2.8,5.15) {(a)};

    \draw[covering] (0,0.45) rectangle (5.6,4.5);
    \node[anchor=south] at (2.8,4.45)
        {$Q_1^{**}\times Q_2^{**}$};

    \path[openset]
        (0.55,1.05)
        -- (3.95,1.05)
        -- (3.95,1.8)
        -- (5.05,1.8)
        -- (5.05,3.65)
        -- (3.75,3.65)
        -- (3.75,4.05)
        -- (1.35,4.05)
        -- (1.35,3.4)
        -- (0.55,3.4)
        -- cycle;

    \node at (4.6,3.2) {$\Omega$};

    \draw[component]
        (1.55,1.3) rectangle (3.45,3.85);
    \draw[component]
        (0.85,2.2) rectangle (4.75,2.85);

    \path[selected]
        (1.65,1.45) rectangle (3.25,1.95);
    \draw[enlarged]
        (0.85,1.2) rectangle (4.05,2.2);

    \node[text=blue!65!black] at (2.45,1.7) {$R$};
    \node[
        anchor=north,
        text=blue!65!black
    ] at (4.05,1.15) {$\tau R$};
\end{scope}

\begin{scope}[xshift=6.6cm]
    \node at (2.65,5.15) {(b)};

    \begin{scope}[shift={(0,3.15)}]
        \path[selected] (0,0) rectangle (0.75,0.45);
        \draw[threshold] (1.1,-0.08) -- (1.1,1.15);
        \draw[threshold] (-0.08,0.75) -- (1.85,0.75);
        \node at (0.9,-0.4) {$1,2$};
    \end{scope}

    \begin{scope}[shift={(3.15,3.15)}]
        \path[selected] (0,0) rectangle (1.75,0.45);
        \draw[threshold] (1.1,-0.08) -- (1.1,1.15);
        \draw[threshold] (-0.08,0.75) -- (1.85,0.75);
        \node at (0.9,-0.4) {$2$};
    \end{scope}

    \begin{scope}[shift={(0,0.9)}]
        \path[selected] (0,0) rectangle (0.75,1.05);
        \draw[threshold] (1.1,-0.08) -- (1.1,1.15);
        \draw[threshold] (-0.08,0.75) -- (1.85,0.75);
        \node at (0.9,-0.4) {$1$};
    \end{scope}

    \begin{scope}[shift={(3.15,0.9)}]
        \path[selected] (0,0) rectangle (1.75,1.05);
        \draw[threshold] (1.1,-0.08) -- (1.1,1.15);
        \draw[threshold] (-0.08,0.75) -- (1.85,0.75);
        \node at (0.9,-0.4) {$\varnothing$};
    \end{scope}
\end{scope}

\end{tikzpicture}%
}
\caption{Schematic geometry of the atoms for $n_1=n_2=1$.
(a) The atom is supported in $\Omega$, whereas each
component $a_R$ is supported in
$\tau R\cap(Q_1^{**}\times Q_2^{**})$.
(b) Horizontal and vertical sides correspond to
$x_1$ and $x_2$, respectively.
The dashed lines mark the thresholds $\rho r_1$
and $\rho r_2$, where $r_i=\eta\ell(Q_i)$.
The labels below the rectangles indicate the variables
in which cancellation is required; $\varnothing$
means that neither cancellation condition is required.}
\label{fig:geometric-atoms}
\end{figure}
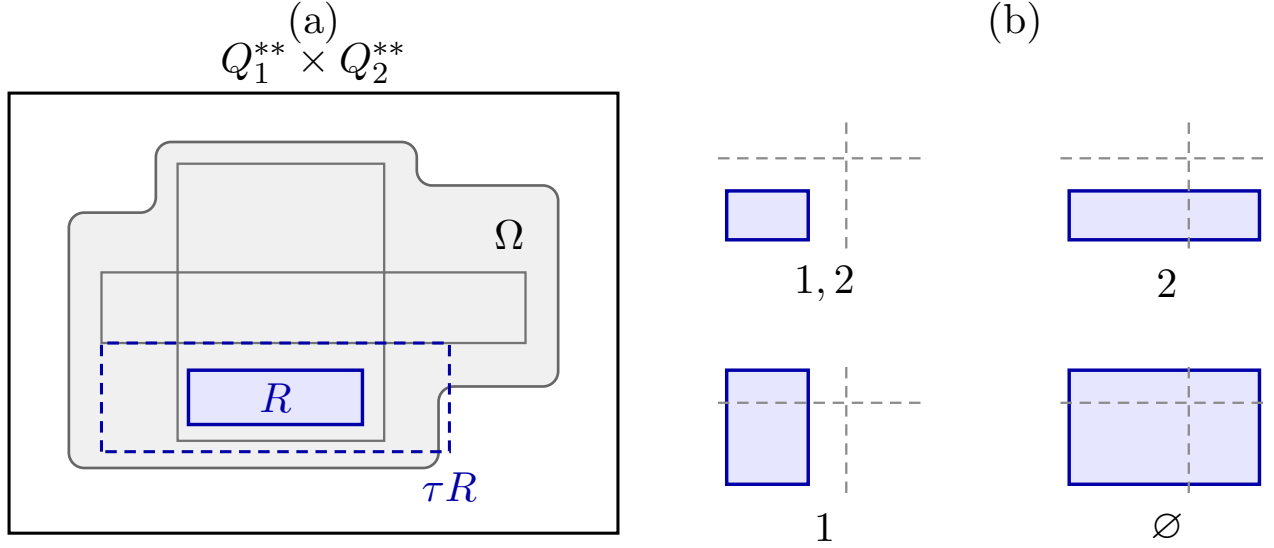

The main contribution is a geometric atomic
characterization at spatially varying scales determined
independently by the two adapted coverings, obtained
through a localization formula for the Hardy norm
and a local product decomposition that keeps both
the associated open sets and every rectangular
component inside the prescribed covering rectangles.
We therefore establish a localized version
of the classical product decomposition, with
constants uniform in both local scales.
This support control allows the local
decompositions of $\phi_{Q_1}\phi_{Q_2}f$
to be combined into atoms adapted to the
coverings. This simultaneous control of scales and supports
allows the decompositions of the localized pieces
$\phi_{Q_1}\phi_{Q_2}f$ to be combined into atoms
satisfying Definition~\ref{def:covering-product-atom}.

The distinction from the existing operator-adapted atomic descriptions is that the atoms in our characterization are specified entirely by geometric conditions: their supports, $L^2$ bounds, and partial cancellation in the individual variables. In particular, their definition does not involve powers of the operators $L_1$ and $L_2$, nor does it require estimates for functions obtained by applying such powers. The dependence on the Schrödinger operators enters only through the admissible coverings and the associated local scales. This provides a direct geometric description of the product Hardy space, rather than an abstract operator-dependent atomic representation.

The comparison with the Schr\"odinger square
function is obtained analytically.
We first derive one-parameter localization
and supported-function estimates from
\cite{DZ_Studia}, using the square function
and maximal function comparison in
\cite{SongYan2018}.
Standard Rademacher randomization and the
Hilbert-space Khintchine inequality then
extend these estimates to sequences and
Hilbert-space-valued functions.
Applying them successively in the two variables
gives the product comparison.
The remaining argument combines this comparison
with the localized atomic construction and
identifies the Hardy space completion with
a subspace of $L^1$.

For $n\geq3$, nonnegative nonzero $L^2_{\rm loc}$ potentials
satisfying a reverse H\"older inequality with
exponent $q>n/2$ admit adapted coverings for
which the conditions hold; see
\cite[Section~8]{DZ_Studia}.
The assumptions also accommodate potentials
outside every reverse H\"older class $RH_q$,
$q>1$, including
\[
    V(x)=\mathbf1_{\{x_1>0\}},
    \qquad
    V(x)=e^{x_1},
    \qquad
    V(x)=\mathbf1_{[-1,1]}(x_1),
    \qquad n\geq1.
\]
For these examples and their adapted coverings,
see \cite[Theorems~2.4 and~2.6, and Section~8]{DZ_Studia}.
The coverings may be uniformly subdivided
to accommodate the fixed enlargements used here.
None of these potentials belongs to any
reverse H\"older class $RH_q$, $q>1$.
The one-dimensional results in
\cite{CzajaZienkiewicz2008,DziubanskiPreisner2011}
also provide examples of the form $V(x)=W(x_1)$,
where $W$ is nonnegative and nonzero;
we retain the local integrability requirements
of our setting.
The two factors may be chosen independently,
so the theorem applies to pairs of operators
with different local scale geometries. We also give explicit details of the support
and convergence arguments needed to identify
the completed Hardy space with its atomic
realization in $L^1$.

The paper is organized as follows.
Section~\ref{sec:setting} contains the assumptions,
the definitions of the Hardy and atomic spaces,
and the statements of the main results.
Section~\ref{sec:local-hardy} develops the
local product estimates and the atomic
decompositions with controlled supports.
Section~\ref{sec:comparison} proves the
comparison with the Schr\"odinger square
function.
Finally, Section~\ref{sec:atomic-decomposition}
establishes the main characterizations and
their consequences.

\section{Setting and main result}
\label{sec:setting}

We specify the operators, coverings, and assumptions
used throughout the paper, and define the associated
product Hardy space.

\subsection{Schr\"odinger operators}
\label{subsec:operators}

Let $n_1,n_2\geq1$.
For $i=1,2$, let
$V_i\in L^2_{\mathrm{loc}}(\mathbb R^{n_i})$
be nonnegative, and let
\[
    L_i=-\Delta_{x_i}+V_i
\]
be the nonnegative self-adjoint operator associated
with the closed quadratic form
\[
    \mathfrak q_i[u]
    =
    \int_{\mathbb R^{n_i}}|\nabla u(x_i)|^2\,dx_i
    +
    \int_{\mathbb R^{n_i}}V_i(x_i)|u(x_i)|^2\,dx_i,
\]
with domain
\[
    \mathcal D(\mathfrak q_i)
    =
    \left\{
        u\in W^{1,2}(\mathbb R^{n_i}):
        V_i^{1/2}u\in L^2(\mathbb R^{n_i})
    \right\}.
\]
We write $K_t^{(i)}=e^{-tL_i}$ and denote its
integral kernel by $k_t^{(i)}(x_i,y_i)$.
The classical heat semigroup is denoted by
$H_t^{(i)}$, with kernel
\[
    h_t^{(i)}(x_i)
    =
    (4\pi t)^{-n_i/2}
    \exp\left(-\frac{|x_i|^2}{4t}\right).
\]
The semigroup $K_t^{(i)}$ is sub-Markovian and satisfies
\begin{equation}\label{eq:heat-domination}
    0\leq k_t^{(i)}(x_i,y_i)
    \leq h_t^{(i)}(x_i-y_i)
\end{equation}
for $t>0$ and almost every $x_i,y_i$. On $\mathbb R^{n_1}\times\mathbb R^{n_2}$,
we consider the two-parameter semigroup
\[
    K_{t_1}^{(1)}\otimes K_{t_2}^{(2)},
    \qquad t_1,t_2>0.
\]
Operators carrying the index $i$ act only in
the variable $x_i$; operators acting in different
variables therefore commute.

\subsection{Admissible coverings and assumptions}
\label{subsec:coverings}

\begin{definition}\label{def:coverings}
For each $i=1,2$, let $\mathcal Q_i$ be a countable
covering of $\mathbb R^{n_i}$ by closed cubes of
positive finite side length, with sides parallel
to the coordinate axes and pairwise disjoint interiors.

For a cube $Q$, let $\ell(Q)$ denote its side length.
For $\lambda>0$, let $\lambda Q$ be the concentric
cube with side length $\lambda\ell(Q)$.
We set
\begin{equation}\label{eq:cube-enlargements}
    Q^*=2Q,\qquad Q^{**}=4Q,\qquad
    Q^{***}=8Q,\qquad Q^{****}=16Q.
\end{equation}
\end{definition}

For each $i=1,2$, we impose the following conditions.

\medskip
\noindent
\textbf{Comparability of neighbouring cubes.}
There exists $C_{F,i}\geq1$ such that
\begin{equation}\tag{$\mathrm F_i$}\label{ass:F}
    Q^{****}\cap(Q')^{****}\neq\varnothing
    \quad\Longrightarrow\quad
    C_{F,i}^{-1}\ell(Q)
    \leq\ell(Q')\leq C_{F,i}\ell(Q),
    \qquad Q,Q'\in\mathcal Q_i.
\end{equation}

\medskip
\noindent
\textbf{Local control of the potential.}
There exist $C_{K,i}>0$ and $\delta_i>0$ such that,
for every $Q\in\mathcal Q_i$ and
$0<t\leq\ell(Q)^2$,
\begin{equation}\tag{$\mathrm K_i$}\label{ass:K}
    \sup_{x_i\in\mathbb R^{n_i}}
    \int_0^{2t}\int_{Q^{***}}
        h_s^{(i)}(x_i-y_i)V_i(y_i)\,dy_i\,ds
    \leq
    C_{K,i}
    \left(\frac{t}{\ell(Q)^2}\right)^{\delta_i}.
\end{equation}

\medskip
\noindent
\textbf{Decay of the semigroup mass.}
There exist $C_{D,i}>0$ and $\varepsilon_i>0$ such that,
for every $Q\in\mathcal Q_i$ and integer $j\geq1$,
\begin{equation}\tag{$\mathrm D_i$}\label{ass:D}
    \operatorname*{ess\,sup}_{y_i\in Q^{****}}
    \int_{\mathbb R^{n_i}}
        k_{2^j\ell(Q)^2}^{(i)}(x_i,y_i)\,dx_i
    \leq C_{D,i}j^{-1-\varepsilon_i}.
\end{equation}

The constants are uniform over the covering cubes
and the indicated parameters.
They may differ between the two factors.

\subsection{Partitions of unity}
\label{subsec:partitions}

The disjointness of the cube interiors and
\eqref{ass:F} imply
\begin{equation}\label{eq:covering-overlap}
    \sum_{Q\in\mathcal Q_i}\mathbf1_{Q^{****}}
    \leq C
    \qquad\text{almost everywhere}.
\end{equation}
For each $i=1,2$, fix a smooth partition of unity
$\{\phi_{i,Q}\}_{Q\in\mathcal Q_i}$ satisfying
\begin{equation}\label{eq:partition-of-unity}
    0\leq\phi_{i,Q}\leq1,\qquad
    \operatorname{supp}\phi_{i,Q}\subset Q^*,
    \qquad
    \sum_{Q\in\mathcal Q_i}\phi_{i,Q}=1,
\end{equation}
and, for every multi-index $\alpha$,
\begin{equation}\label{eq:partition-derivative-bounds}
    |\partial^\alpha\phi_{i,Q}(x_i)|
    \leq C_{\alpha,i}\ell(Q)^{-|\alpha|}.
\end{equation}
Such partitions can be obtained by normalizing
smooth cutoffs equal to one on $Q$ and supported
in $Q^*$; \eqref{ass:F} and
\eqref{eq:covering-overlap} give the uniform
derivative bounds.

When $Q\in\mathcal Q_i$, we abbreviate
$\phi_{i,Q}$ to $\phi_Q$.
In particular,
\begin{equation}\label{eq:product-partition-of-unity}
    \sum_{Q_1\in\mathcal Q_1}
    \sum_{Q_2\in\mathcal Q_2}
        \phi_{Q_1}(x_1)\phi_{Q_2}(x_2)=1.
\end{equation}

\subsection{Square functions and the product Hardy space}
\label{subsec:hardy-space}

For $x_i\in\mathbb R^{n_i}$, set
\[
    \Gamma_i(x_i)
    =
    \left\{
        (y_i,t_i)\in\mathbb R^{n_i}\times(0,\infty):
        |y_i-x_i|<t_i
    \right\}.
\]
The one-parameter conical square function is
\begin{equation}\label{eq:one-parameter-square-function}
    S_{L_i}g(x_i)
    =
    \left(
        \iint_{\Gamma_i(x_i)}
            |t_i^2L_i e^{-t_i^2L_i}g(y_i)|^2
            \frac{dy_i\,dt_i}{t_i^{n_i+1}}
    \right)^{1/2},
    \qquad g\in L^2(\mathbb R^{n_i}).
\end{equation}
For $f\in L^2(\mathbb R^{n_1}\times\mathbb R^{n_2})$,
define
\begin{equation}\label{eq:product-square-function}
\begin{split}
    S_{L_1,L_2}f(x_1,x_2)
    =
    \Bigg(
        \iint_{\Gamma_1(x_1)}
        \iint_{\Gamma_2(x_2)}
        &\left|
            (t_1^2L_1e^{-t_1^2L_1})
            (t_2^2L_2e^{-t_2^2L_2})
            f(y_1,y_2)
        \right|^2
        \frac{dy_2\,dt_2}{t_2^{n_2+1}}
        \frac{dy_1\,dt_1}{t_1^{n_1+1}}
    \Bigg)^{1/2}.
\end{split}
\end{equation}

Both square functions are bounded on $L^2$:
\begin{equation}\label{eq:schrodinger-square-L2}
    \|S_{L_i}g\|_2\leq C\|g\|_2,
    \qquad
    \|S_{L_1,L_2}f\|_2\leq C\|f\|_2.
\end{equation}
Indeed, integrating first in the cone vertices
and then applying the spectral theorem reduces
these estimates to
\[
    \int_0^\infty
        (t^2\lambda)^2e^{-2t^2\lambda}\,\frac{dt}{t}
    =\frac18,
    \qquad \lambda>0.
\]
Moreover, $\ker L_i=\{0\}$: if $L_i u=0$, then
$\mathfrak q_i[u]=0$, so $\nabla u=0$ and
$u=0$ in $L^2(\mathbb R^{n_i})$.
The same spectral calculation therefore shows
that $S_{L_1,L_2}f=0$ almost everywhere implies
$f=0$.

\begin{definition}\label{def:schrodinger-product-hardy}
Let
\[
    \mathcal H^1_{L_1,L_2}
    =
    \left\{
        f\in L^2(\mathbb R^{n_1}\times\mathbb R^{n_2}):
        S_{L_1,L_2}f\in L^1
    \right\}.
\]
The space $H^1_{L_1,L_2}$ is the completion of
$\mathcal H^1_{L_1,L_2}$ with respect to
\[
    \|f\|_{H^1_{L_1,L_2}}
    =
    \|S_{L_1,L_2}f\|_1.
\]
\end{definition}

We will prove that this completion is canonically
embedded in $L^1$; see
\eqref{eq:schrodinger-completion-L1-embedding}.

Unless otherwise stated, implicit constants may
depend on the dimensions, the constants and exponents
in \eqref{ass:F}, \eqref{ass:K}, and \eqref{ass:D},
and the fixed auxiliary functions.
They are independent of the covering cubes,
the local scales, and the functions being estimated.

\subsection{Atoms and the main results}
\label{subsec:main-results}

Fix $\tau>1$ and $\rho>0$.
Let $\eta\in(0,1)$ be sufficiently small;
the required choice will be specified in
Corollary~\ref{cor:atoms-inside-enlargements}.
These parameters remain fixed throughout the paper.

For any cube $I$ and $\tau>0$, $\tau I$ denotes
the cube with the same centre as $I$ and side
length $\tau\ell(I)$. For $i=1,2$, let $\mathcal D_i$ be the standard
dyadic grid in $\mathbb R^{n_i}$.
Given $r_1,r_2>0$ and an open set $\Omega$ of
finite positive measure, let
$\mathcal M_{r_1,r_2}(\Omega)$ consist of all
rectangles
\[
    R=I\times J,
    \qquad I\in r_1\mathcal D_1,
    \qquad J\in r_2\mathcal D_2,
\]
which are contained in $\Omega$ and maximal
with respect to inclusion.
Here $r_i\mathcal D_i=\{r_iI:I\in\mathcal D_i\}$,
and maximality is taken within this fixed
product grid: no strictly larger rectangle
from the same grid containing $R$ is contained
in $\Omega$.
For $R=I\times J$, write
$\tau R=(\tau I)\times(\tau J)$.

\begin{definition}\label{def:covering-product-atom}
A function
$a\in L^2(\mathbb R^{n_1}\times\mathbb R^{n_2})$
is an atom associated with
$\mathcal Q_1\times\mathcal Q_2$ if there exist
$Q_i\in\mathcal Q_i$, $i=1,2$, and an open set
\[
    \Omega\subset Q_1^{**}\times Q_2^{**},
    \qquad 0<|\Omega|<\infty,
\]
such that, with $r_i=\eta\ell(Q_i)$,
the following conditions hold.

\begin{enumerate}
\item
The function $a$ vanishes almost everywhere
outside $\Omega$, and
\[
    \|a\|_2\leq|\Omega|^{-1/2}.
\]

\item
There are functions
$a_R\in L^2(\mathbb R^{n_1}\times\mathbb R^{n_2})$,
indexed by
$R\in\mathcal M_{r_1,r_2}(\Omega)$, such that
\[
    a=\sum_{R\in\mathcal M_{r_1,r_2}(\Omega)}a_R
    \quad\text{in }L^2,
    \qquad
    \sum_{R\in\mathcal M_{r_1,r_2}(\Omega)}
        \|a_R\|_2^2
    \leq|\Omega|^{-1}.
\]
Each $a_R$ vanishes almost everywhere outside $
    \tau R\cap(Q_1^{**}\times Q_2^{**}).
$

\item
For $R=I\times J$, the components satisfy
\[
    \int_{\mathbb R^{n_1}}a_R(x_1,x_2)\,dx_1=0
    \quad\text{for almost every }x_2
    \quad\text{if }\ell(I)<\rho r_1,
\]
and
\[
    \int_{\mathbb R^{n_2}}a_R(x_1,x_2)\,dx_2=0
    \quad\text{for almost every }x_1
    \quad\text{if }\ell(J)<\rho r_2.
\]
No cancellation in a variable is required
when the corresponding side length is $ \geq \rho r_i$.
\end{enumerate}
\end{definition}

Let us note that every atom satisfies
\begin{equation}\label{eq:covering-product-atom-L1}
    \|a\|_1
    \leq|\Omega|^{1/2}\|a\|_2
    \leq1.
\end{equation}

\begin{definition}\label{def:covering-product-atomic-space}
The space
$H^1_{\mathrm{at}}(\mathcal Q_1,\mathcal Q_2)$
consists of all $f\in L^1$ admitting a representation
\[
    f=\sum_{\nu=1}^{\infty}\lambda_\nu a_\nu
    \quad\text{in }L^1,
    \qquad
    \sum_{\nu=1}^{\infty}|\lambda_\nu|<\infty,
\]
where each $a_\nu$ satisfies
Definition~\ref{def:covering-product-atom}.
Its norm is
\[
    \|f\|_{H^1_{\mathrm{at}}(\mathcal Q_1,\mathcal Q_2)}
    =
    \inf\left\{
        \sum_{\nu=1}^{\infty}|\lambda_\nu|:
        f=\sum_{\nu=1}^{\infty}\lambda_\nu a_\nu
        \text{ in }L^1
    \right\}.
\]
The fixed parameters $\eta,\rho,\tau$ are
suppressed in the notation.
\end{definition}

By \eqref{eq:covering-product-atom-L1}, every
atomic series with summable coefficients
converges absolutely in $L^1$, and
\[
    \|f\|_1
    \leq
    \|f\|_{H^1_{\mathrm{at}}(\mathcal Q_1,\mathcal Q_2)}.
\]

\begin{theorem}\label{thm:main-atomic-characterization}
Let $L_i$ and $\mathcal Q_i$, $i=1,2$, be as
in Subsections~\ref{subsec:operators}
and~\ref{subsec:coverings}, and assume
\eqref{ass:F}, \eqref{ass:K}, and \eqref{ass:D}.
For fixed $\tau>1$ and $\rho>0$, choose $\eta>0$
sufficiently small as above.

Then the space $H^1_{L_1,L_2}$ from
Definition~\ref{def:schrodinger-product-hardy}
is canonically embedded in $L^1$, and
\[
    H^1_{L_1,L_2}
    =
    H^1_{\mathrm{at}}(\mathcal Q_1,\mathcal Q_2)
\]
with equivalent norms.

More precisely, every $f\in H^1_{L_1,L_2}$
admits an atomic representation
\[
    f=\sum_{\nu=1}^{\infty}\lambda_\nu a_\nu
\]
converging in both $H^1_{L_1,L_2}$ and $L^1$, with
\[
    \sum_{\nu=1}^{\infty}|\lambda_\nu|
    \leq C\|f\|_{H^1_{L_1,L_2}}.
\]
Conversely, every series of atoms from
Definition~\ref{def:covering-product-atom}
with summable coefficients converges in both
spaces and satisfies
\[
    \left\|
        \sum_{\nu=1}^{\infty}\lambda_\nu a_\nu
    \right\|_{H^1_{L_1,L_2}}
    \leq C\sum_{\nu=1}^{\infty}|\lambda_\nu|.
\]
The constants depend only on the dimensions,
the constants and exponents in the assumptions,
the fixed auxiliary functions, and
$\eta,\rho,\tau$.
\end{theorem}

We also obtain a characterization by classical
local product Hardy norms.
Here $h^1_{r_1,r_2}$ denotes the local product
Hardy space at scales $r_1,r_2$, defined in
Section~\ref{sec:local-hardy}.

\begin{theorem}\label{thm:main-localization}
Under the assumptions of
Theorem~\ref{thm:main-atomic-characterization},
a function $f\in L^1$ belongs to $H^1_{L_1,L_2}$
if and only if
\[
    \sum_{Q_1\in\mathcal Q_1}
    \sum_{Q_2\in\mathcal Q_2}
        \|\phi_{Q_1}\phi_{Q_2}f\|
            _{h^1_{\ell(Q_1),\ell(Q_2)}}<\infty.
\]
Moreover,
\[
    \|f\|_{H^1_{L_1,L_2}}
    \asymp
    \sum_{Q_1\in\mathcal Q_1}
    \sum_{Q_2\in\mathcal Q_2}
        \|\phi_{Q_1}\phi_{Q_2}f\|
            _{h^1_{\ell(Q_1),\ell(Q_2)}}.
\]
The constants are independent of $f$.
\end{theorem}

\section{Local product Hardy spaces}
\label{sec:local-hardy}

We collect the local Hardy space estimates needed
to prove Theorems~\ref{thm:main-atomic-characterization}
and~\ref{thm:main-localization}.
We first establish scale comparison and cutoff
estimates, and then construct local atomic
decompositions with controlled supports.
The final corollary places these atoms inside
the enlarged covering cubes from
Definition~\ref{def:covering-product-atom}.

\subsection{Definitions and basic properties}
\label{subsec:local-product-hardy}

For $i=1,2$, fix
$\psi^{(i)}_0\in C^\infty_c(\mathbb R^{n_i})$
with $\int\psi^{(i)}_0=1$, and set
\[
    \psi^{(i)}_j(x_i)
    =
    2^{jn_i}\psi^{(i)}_0(2^jx_i)
    -
    2^{(j-1)n_i}\psi^{(i)}_0(2^{j-1}x_i),
    \qquad j\geq1.
\]
Thus $\int\psi^{(i)}_j=0$ for $j\geq1$.
For $r>0$, define
\[
    \psi^{(i)}_{j,r}(x_i)
    =
    r^{-n_i}\psi^{(i)}_j(x_i/r),
    \qquad j\geq0.
\]
We write $*_i$ for convolution in $x_i$ only.

For functions on $\mathbb R^{n_i}$, set
\begin{equation}\label{eq:B}
    \mathcal B_{i,r}g
    =
    \bigl(\psi^{(i)}_{j,r}*g\bigr)_{j\geq0},
    \qquad
    \|g\|_{h^1_r}
    =
    \|\mathcal B_{i,r}g\|_{L^1(\ell^2)}.
\end{equation}
The space $h^1_r(\mathbb R^{n_i})$ consists of
the functions $g\in L^1(\mathbb R^{n_i})$
for which this norm is finite.
When applied to a function of both spatial
variables, $\mathcal B_{i,r}$ acts in $x_i$ only.

For $r_1,r_2>0$, define
\begin{equation}\label{eq:local-product-square}
    S^{\mathrm{loc}}_{r_1,r_2}f
    =
    \left(
        \sum_{j,k\geq0}
        \left|
            \psi^{(1)}_{j,r_1}*_1
            \psi^{(2)}_{k,r_2}*_2f
        \right|^2
    \right)^{1/2}.
\end{equation}
The local product Hardy space is
\[
    h^1_{r_1,r_2}
    =
    \left\{
        f\in L^1(\mathbb R^{n_1}\times\mathbb R^{n_2}):
        S^{\mathrm{loc}}_{r_1,r_2}f\in L^1
    \right\},
\]
with norm
\begin{equation}\label{eq:local-product-Hardy-norm}
    \|f\|_{h^1_{r_1,r_2}}
    =
    \|S^{\mathrm{loc}}_{r_1,r_2}f\|_1.
\end{equation}
The indices $j=0$ and $k=0$ retain the
low-frequency contribution in each variable.

By \cite[Theorem~1.2]{DingLuZhu2019}, different
choices of $\psi^{(i)}_0$ give equivalent norms.
The resulting space agrees with the space
of tempered distributions introduced in
\cite[Definition~1.1]{DingLuZhu2019}.
Indeed, the atomic decomposition in
\cite[Theorem~1.5]{DingLuZhu2019}, with $p=1$,
represents every element of that space by
an $L^1$-convergent series, since the atoms
have uniformly bounded $L^1$ norms and the
coefficients are summable.
Finally, \cite[Theorem~1.3]{DingLuZhu2019}
gives density of Schwartz functions.
These results are stated at unit scales;
independent dilations give the assertions
for arbitrary $r_1,r_2>0$.

Independent dilations reduce arbitrary scales
to unit scales. Namely, if
\[
    \widetilde f(u_1,u_2)
    =
    r_1^{n_1}r_2^{n_2}f(r_1u_1,r_2u_2),
\]
then
\begin{equation}\label{eq:local-product-dilation}
    \|\widetilde f\|_{h^1_{1,1}}
    =
    \|f\|_{h^1_{r_1,r_2}}.
\end{equation}
In particular, the constants in the equivalence
of the local norms can be chosen independently
of $r_1,r_2$.

\begin{lemma}\label{lem:local-heat-characterization}
For $i=1,2$, $r>0$, and
$g\in L^1(\mathbb R^{n_i})$,
\begin{equation}\label{eq:local-heat-characterization}
    C^{-1}\|\mathcal B_{i,r}g\|_{L^1(\ell^2)}
    \leq
    \left\|
        \sup_{0<t\leq r^2}|H_t^{(i)}g|
    \right\|_1
    \leq
    C\|\mathcal B_{i,r}g\|_{L^1(\ell^2)}.
\end{equation}
The inequalities are understood with values in
$[0,\infty]$, and $C$ is independent of $r$ and $g$.
\end{lemma}

\begin{proof}
At scale $r=1$, Goldberg's local maximal
characterization \cite[Theorem~1]{Goldberg},
applied to the Gaussian
$
    \varphi(x)
    =
    (4\pi)^{-n_i/2}\exp(-|x|^2/4)
$,
gives
\[
    \|g\|_{h^1(\mathbb R^{n_i})}
    \asymp
    \left\|
        \sup_{0<t\leq1}|H_t^{(i)}g|
    \right\|_1.
\]
Indeed, for
$\varphi_s(x)=s^{-n_i}\varphi(x/s)$,
we have $\varphi_s*g=H_{s^2}^{(i)}g$.

The one-parameter Littlewood--Paley
characterization \cite{Rychkov2001}, recalled in
\cite[p.~354]{DingLuZhu2019}, gives
\[
    \|g\|_{h^1(\mathbb R^{n_i})}
    \asymp
    \|\mathcal B_{i,1}g\|_{L^1(\ell^2)}.
\]
Combining these estimates proves the assertion
at unit scale.

For general $r>0$, apply the unit-scale estimate
to $g(r\,\cdot)$.
The identities
$
    \mathcal B_{i,1}\bigl(g(r\,\cdot)\bigr)(x)
    =
    (\mathcal B_{i,r}g)(rx)
$
and
$
    H_t^{(i)}\bigl(g(r\,\cdot)\bigr)(x)
    =
    (H_{r^2t}^{(i)}g)(rx)
$
show that both $L^1$ norms acquire the factor
$r^{-n_i}$.
Cancelling this factor proves
\eqref{eq:local-heat-characterization}.
\end{proof}

Since $H_t^{(i)}g\to g$ in $L^1$ as $t\downarrow0$,
\eqref{eq:local-heat-characterization} implies
\begin{equation}\label{eq:local-L1-embedding}
    \|g\|_1
    \leq
    \left\|
        \sup_{0<t\leq r^2}|H_t^{(i)}g|
    \right\|_1
    \leq C\|g\|_{h^1_r},
    \qquad g\in h^1_r,
\end{equation}
with $C$ independent of $r$.

We will use the following extension of scalar-input
estimates to sequences.

\begin{lemma}\label{lem:sequence-extension}
Let $A$ and $B$ be linear operators on a common
vector space of scalar-valued functions, taking
values in separable Hilbert spaces $H_A$ and $H_B$.
Suppose that
\[
    \|Bg\|_{L^1(H_B)}
    \leq C_0\|Ag\|_{L^1(H_A)}
\]
whenever the right-hand side is finite.
Then
\begin{equation}\label{eq:sequence-extension}
    \int
        \left(
            \sum_\nu\|Bg_\nu(x)\|_{H_B}^2
        \right)^{1/2}dx
    \leq
    CC_0\int
        \left(
            \sum_\nu\|Ag_\nu(x)\|_{H_A}^2
        \right)^{1/2}dx
\end{equation}
for every finite or countable family in the
common domain for which the right-hand side
is finite.
The constant $C$ is independent of the number
of inputs and of the Hilbert spaces.
\end{lemma}

\begin{proof}
Let
\[
    r_\nu(\tau)
    =
    \operatorname{sgn}\bigl(\sin(2^\nu\pi\tau)\bigr),
    \qquad \nu\geq1,\quad 0\leq\tau\leq1,
\]
be the Rademacher functions; see
\cite[Appendix~C.1]{GrafakosClassical}.
Their values at the zeros of the sine function
are immaterial.

We first verify the Hilbert-valued form of
Khintchine's inequality:
\begin{equation}\label{eq:sequence-khintchine}
    \int_0^1
        \left\|
            \sum_{\nu=1}^N r_\nu(\tau)v_\nu
        \right\|_H\,d\tau
    \asymp
    \left(
        \sum_{\nu=1}^N\|v_\nu\|_H^2
    \right)^{1/2}.
\end{equation}
Choose an orthonormal basis of $H$ and write
$v_{\nu,j}$ for the coordinates of $v_\nu$.
The scalar Khintchine inequality
\cite[Appendix~C.2]{GrafakosClassical} gives
\[
    \left(
        \sum_{\nu=1}^N|v_{\nu,j}|^2
    \right)^{1/2}
    \leq
    C\int_0^1
        \left|
            \sum_{\nu=1}^N r_\nu(\tau)v_{\nu,j}
        \right|\,d\tau.
\]
Taking the $\ell^2$ norm over the coordinates
and applying Minkowski's integral inequality,
we obtain
\begin{align*}
    \left(
        \sum_{\nu=1}^N\|v_\nu\|_H^2
    \right)^{1/2}
    &\leq
    C\left(
        \sum_j
        \left[
            \int_0^1
                \left|
                    \sum_{\nu=1}^N
                        r_\nu(\tau)v_{\nu,j}
                \right|\,d\tau
        \right]^2
    \right)^{1/2}\\
    &\leq
    C\int_0^1
        \left(
            \sum_j
            \left|
                \sum_{\nu=1}^N
                    r_\nu(\tau)v_{\nu,j}
            \right|^2
        \right)^{1/2}d\tau=
    C\int_0^1
        \left\|
            \sum_{\nu=1}^N r_\nu(\tau)v_\nu
        \right\|_H\,d\tau.
\end{align*}
For infinitely many coordinates, the argument
is first applied to finite truncations and
then extended by monotone convergence.

For the opposite inequality, Cauchy--Schwarz
in $\tau$ gives
\[
    \int_0^1
        \left\|
            \sum_{\nu=1}^N r_\nu(\tau)v_\nu
        \right\|_H\,d\tau
    \leq
    \left(
        \int_0^1
        \left\|
            \sum_{\nu=1}^N r_\nu(\tau)v_\nu
        \right\|_H^2\,d\tau
    \right)^{1/2}.
\]
By the orthogonality relation
$
    \int_0^1r_\nu(\tau)r_\mu(\tau)\,d\tau
    =\delta_{\nu\mu},
$
the square of the last expression is
\[
    \sum_{\nu,\mu=1}^N
        \langle v_\nu,v_\mu\rangle_H
        \int_0^1r_\nu(\tau)r_\mu(\tau)\,d\tau
    =
    \sum_{\nu=1}^N\|v_\nu\|_H^2.
\]
This proves \eqref{eq:sequence-khintchine}.

Now apply the assumed operator estimate to
$\sum_{\nu=1}^N r_\nu(\tau)g_\nu$ and integrate
over $\tau\in[0,1]$.
Linearity, Tonelli's theorem, and
\eqref{eq:sequence-khintchine} yield
\begin{align*}
   &\int
        \left(
            \sum_{\nu=1}^N\|Bg_\nu(x)\|_{H_B}^2
        \right)^{1/2}dx
    \leq
    C\int_0^1
        \left\|
            B\left(
                \sum_{\nu=1}^N r_\nu(\tau)g_\nu
            \right)
        \right\|_{L^1(H_B)}d\tau\\
    \leq &
    CC_0\int_0^1
        \left\|
            A\left(
                \sum_{\nu=1}^N r_\nu(\tau)g_\nu
            \right)
        \right\|_{L^1(H_A)}d\tau\leq
    CC_0\int
        \left(
            \sum_{\nu=1}^N\|Ag_\nu(x)\|_{H_A}^2
        \right)^{1/2}dx.
\end{align*}
The constant is independent of $N$.
Monotone convergence proves the assertion
for countable families.
\end{proof}

\begin{lemma}\label{lem:local-product-L1-embedding}
For every $r_1,r_2>0$ and $g\in h^1_{r_1,r_2}$,
\begin{equation}\label{eq:product-local-L1-embedding}
    \|g\|_1\leq C\|g\|_{h^1_{r_1,r_2}},
\end{equation}
where $C$ is independent of $r_1,r_2$ and $g$.
\end{lemma}

\begin{proof}
Apply Lemma~\ref{lem:sequence-extension} to
\eqref{eq:local-L1-embedding}, with the identity
operator on the left.
For every finite family
$q_0,\ldots,q_J\in h^1_{r_2}$, this gives
\[
    \int_{\mathbb R^{n_2}}
        \left(
            \sum_{j=0}^J|q_j(x_2)|^2
        \right)^{1/2}dx_2
    \leq
    C\int_{\mathbb R^{n_2}}
        \left(
            \sum_{j=0}^J
            \|\mathcal B_{2,r_2}q_j(x_2)\|_{\ell^2}^2
        \right)^{1/2}dx_2.
\]

For almost every fixed $x_1$, consider
$
    q_j(x_2)
    =
    \bigl(\psi^{(1)}_{j,r_1}*_1g\bigr)(x_1,x_2)
$.
Young's inequality gives
\[
    \|\psi^{(1)}_{j,r_1}*_1g\|_1
    \leq
    \|\psi^{(1)}_{j,r_1}\|_1\|g\|_1,
\]
so $q_j\in L^1(\mathbb R^{n_2})$ for almost
every $x_1$.
Moreover,
\[
    \|\mathcal B_{2,r_2}q_j(x_2)\|_{\ell^2}
    =
    \left(
        \sum_{k\geq0}
        \left|
            \psi^{(1)}_{j,r_1}*_1
            \psi^{(2)}_{k,r_2}*_2g(x_1,x_2)
        \right|^2
    \right)^{1/2}
    \leq S^{\mathrm{loc}}_{r_1,r_2}g(x_1,x_2).
\]
By Fubini's theorem, the latter function is
integrable in $x_2$ for almost every $x_1$.
Thus all $q_j$ belong to $h^1_{r_2}$ outside
a common exceptional null set. For almost every fixed $x_1$, apply the sequence
estimate to $q_0,\ldots,q_J$.
By the definition of $\mathcal B_{2,r_2}$,
\[
    \sum_{j=0}^J
        \|\mathcal B_{2,r_2}q_j(x_2)\|_{\ell^2}^2
    =
    \sum_{j=0}^J\sum_{k\geq0}
        \left|
            \psi^{(1)}_{j,r_1}*_1
            \psi^{(2)}_{k,r_2}*_2g(x_1,x_2)
        \right|^2.
\]
This is bounded by
$\bigl(S^{\mathrm{loc}}_{r_1,r_2}g(x_1,x_2)\bigr)^2$,
because the latter includes all indices $j\geq0$.
Consequently,
\[
    \int_{\mathbb R^{n_2}}
        \left(
            \sum_{j=0}^J
            |\psi^{(1)}_{j,r_1}*_1g(x_1,x_2)|^2
        \right)^{1/2}dx_2
    \leq
    C\int_{\mathbb R^{n_2}}
        S^{\mathrm{loc}}_{r_1,r_2}g(x_1,x_2)\,dx_2.
\]
Integrating in $x_1$ gives a bound independent
of $J$. The expression on the left increases
as $J\to\infty$, so monotone convergence yields
\begin{equation}\label{eq:partial-local-square-L1}
\begin{split}
    &\int_{\mathbb R^{n_2}}
     \int_{\mathbb R^{n_1}}
        \left(
            \sum_{j\geq0}
            |\psi^{(1)}_{j,r_1}*_1g(x_1,x_2)|^2
        \right)^{1/2}dx_1\,dx_2\leq C\|g\|_{h^1_{r_1,r_2}}.
\end{split}
\end{equation}

We now apply the one-parameter $L^1$ estimate
in the first variable.
Since $g\in L^1$, Fubini's theorem gives
$
    g(\cdot,x_2)\in L^1(\mathbb R^{n_1})
    \quad\text{for almost every }x_2
$.
For such a fixed $x_2$, convolution in $x_1$
is precisely convolution of the function
$g(\cdot,x_2)$:
\[
    \bigl(\psi^{(1)}_{j,r_1}*g(\cdot,x_2)\bigr)(x_1)
    =
    \bigl(\psi^{(1)}_{j,r_1}*_1g\bigr)(x_1,x_2).
\]
Thus the inner integral in
\eqref{eq:partial-local-square-L1} is exactly
$\|g(\cdot,x_2)\|_{h^1_{r_1}}$.
It is finite for almost every $x_2$, and
\eqref{eq:local-L1-embedding} therefore gives
\[
    \int_{\mathbb R^{n_1}}|g(x_1,x_2)|\,dx_1
    \leq C\|g(\cdot,x_2)\|_{h^1_{r_1}}.
\]
Finally, integrate in $x_2$ and use
\eqref{eq:partial-local-square-L1}:
\[
    \|g\|_1
    \leq
    C\int_{\mathbb R^{n_2}}
        \|g(\cdot,x_2)\|_{h^1_{r_1}}\,dx_2
    \leq C\|g\|_{h^1_{r_1,r_2}}.
\]
This proves \eqref{eq:product-local-L1-embedding}.
\end{proof}

\begin{lemma}\label{lem:local-Hardy-Fatou}
Fix $r_1,r_2>0$.
If $f_m\to f$ in $L^1$, then
\[
    \|f\|_{h^1_{r_1,r_2}}
    \leq
    \liminf_{m\to\infty}
        \|f_m\|_{h^1_{r_1,r_2}},
\]
where both sides are allowed to be infinite.
The analogous assertion holds for the
one-parameter local Hardy norm.
\end{lemma}

\begin{proof}
For fixed $j,k$, boundedness of the kernels gives
\[
    \left\|
        \bigl(
            \psi^{(1)}_{j,r_1}\otimes
            \psi^{(2)}_{k,r_2}
        \bigr)*(f_m-f)
    \right\|_\infty
    \leq
    \|\psi^{(1)}_{j,r_1}\|_\infty
    \|\psi^{(2)}_{k,r_2}\|_\infty
    \|f_m-f\|_1
    \longrightarrow0.
\]
Apply Fatou's lemma first to the square sum
in \eqref{eq:local-product-square} and then
to the spatial integral.
The one-parameter assertion follows in the
same way.
\end{proof}

The space $h^1_{r_1,r_2}$ is complete.
Indeed, a Cauchy sequence converges in $L^1$
by \eqref{eq:product-local-L1-embedding}.
Applying Lemma~\ref{lem:local-Hardy-Fatou}
to its differences then gives convergence
in the local Hardy norm.

\subsection{Changes of scale and cutoff functions}
\label{subsec:local-scales-cutoffs}

We compare the local Hardy norms at different
scales and estimate multiplication by smooth
cutoff functions.
The one-parameter estimates will be transferred
to the product setting by
Lemma~\ref{lem:sequence-extension}.

\begin{lemma}[Changes of scale and cutoff functions]
\label{lem:local-scales-cutoffs}
The following assertions hold.

\begin{enumerate}
\item
If $0<a_i\leq r_i$ for $i=1,2$, then
\begin{equation}\label{eq:local-scale-comparison}
    \|f\|_{h^1_{a_1,a_2}}
    \leq C\|f\|_{h^1_{r_1,r_2}},
\end{equation}
and
\begin{equation}\label{eq:local-scale-comparison-reverse}
    \|f\|_{h^1_{r_1,r_2}}
    \leq
    C\prod_{i=1}^2
        \left(1+\log\frac{r_i}{a_i}\right)
    \|f\|_{h^1_{a_1,a_2}}.
\end{equation}

\item
If $\chi_i\in C^\infty_c(\mathbb R^{n_i})$,
$i=1,2$, then
\begin{equation}\label{eq:local-product-cutoff}
    \|\chi_1(x_1)\chi_2(x_2)f\|_{h^1_{r_1,r_2}}
    \leq
    C\prod_{i=1}^2
        \left(
            \|\chi_i\|_\infty
            +r_i\|\nabla\chi_i\|_\infty
        \right)
    \|f\|_{h^1_{r_1,r_2}}.
\end{equation}
\end{enumerate}

Each estimate holds for every function in
the space on its right-hand side.
The constant depends only on the dimensions
and the fixed kernels.
\end{lemma}

\begin{proof}
\medskip
\noindent
\textbf{Step 1: One-parameter changes of scale.}

We omit the index $i$ and write
\[
    \mathcal B_rg=(\psi_{j,r}*g)_{j\geq0},
    \qquad
    m_rg=\sup_{0<t\leq r^2}|H_tg|,
\]
where $\mathcal B_r$ denotes the operator
$\mathcal B_{i,r}$ from \eqref{eq:B}
in the variable under consideration.
For $0<a\leq r$, we have $m_ag\leq m_rg$.
Moreover,
\[
    \int_{\mathbb R^n}
        \sup_{b\leq t\leq2b}h_t(x-y)\,dx
    \leq C,
    \qquad b>0,
\]
uniformly in $y$.
Dividing $[a^2,r^2]$ into dyadic intervals gives
\[
    \|m_rg\|_1
    \leq
    \|m_ag\|_1
    +
    C\left(1+\log\frac{r}{a}\right)\|g\|_1.
\]
Using \eqref{eq:local-heat-characterization}
and the $L^1$ estimate
\eqref{eq:local-L1-embedding} at scale $a$,
$
    \|g\|_1
    \leq C\|\mathcal B_ag\|_{L^1(\ell^2)},
$,
we obtain
\begin{equation}\label{eq:one-parameter-scale}
\begin{split}
    \|\mathcal B_ag\|_{L^1(\ell^2)}
    &\leq C\|\mathcal B_rg\|_{L^1(\ell^2)},\\
    \|\mathcal B_rg\|_{L^1(\ell^2)}
    &\leq
    C\left(1+\log\frac{r}{a}\right)
    \|\mathcal B_ag\|_{L^1(\ell^2)}.
\end{split}
\end{equation}

\medskip
\noindent
\textbf{Step 2: The one-parameter cutoff estimate.}

Write
\[
    \mathcal B_r(\chi g)
    =
    \chi\,\mathcal B_rg+[\mathcal B_r,\chi]g,
\]
where $[\mathcal B_r,\chi]g:=\mathcal B_r(\chi g)-\chi\,\mathcal B_rg$. Its $j$th component is
\[
    C_jg(x)
    =
    \int_{\mathbb R^n}
        \psi_{j,r}(x-y)
        \bigl(\chi(y)-\chi(x)\bigr)g(y)\,dy.
\]
Since
\[
    \left(\sum_{j\geq0}|C_jg(x)|^2\right)^{1/2}
    \leq\sum_{j\geq0}|C_jg(x)|,
\]
we have
$
    \|[\mathcal B_r,\chi]g\|_{L^1(\ell^2)}
    \leq\sum_{j\geq0}\|C_jg\|_1
$.
Using
$|\chi(y)-\chi(x)|
\leq\|\nabla\chi\|_\infty|x-y|$
and Tonelli's theorem,
\begin{align*}
    \|C_jg\|_1
    &\leq
    \|\nabla\chi\|_\infty
    \int_{\mathbb R^n}|g(y)|
    \left(
        \int_{\mathbb R^n}
            |x-y|\,|\psi_{j,r}(x-y)|\,dx
    \right)dy=
    \|\nabla\chi\|_\infty\|g\|_1
    \int_{\mathbb R^n}|z|\,|\psi_{j,r}(z)|\,dz.
\end{align*}
The definition of the kernels gives
\[
    \sum_{j\geq0}
        \int_{\mathbb R^n}|z|\,|\psi_{j,r}(z)|\,dz
    \leq Cr.
\]
Indeed, a change of variables yields
$
    \int_{\mathbb R^n}|z|\,|\psi_{0,r}(z)|\,dz
    =
    r\int_{\mathbb R^n}|z|\,|\psi_0(z)|\,dz
$,
while, for $j\geq1$, the two terms defining
$\psi_{j,r}$ give
\[
    \int_{\mathbb R^n}|z|\,|\psi_{j,r}(z)|\,dz
    \leq
    r\bigl(2^{-j}+2^{-(j-1)}\bigr)
    \int_{\mathbb R^n}|z|\,|\psi_0(z)|\,dz.
\]
Summing the geometric series proves the estimate.
Therefore,
\[
    \|[\mathcal B_r,\chi]g\|_{L^1(\ell^2)}
    \leq Cr\|\nabla\chi\|_\infty\|g\|_1.
\]
Together with \eqref{eq:local-L1-embedding},
this proves
\begin{equation}\label{eq:one-parameter-cutoff}
    \|\mathcal B_r(\chi g)\|_{L^1(\ell^2)}
    \leq
    C\bigl(\|\chi\|_\infty+r\|\nabla\chi\|_\infty\bigr)
    \|\mathcal B_rg\|_{L^1(\ell^2)}.
\end{equation}

\medskip
\noindent
\textbf{Step 3: Changes of scale in two variables.}

Lemma~\ref{lem:sequence-extension}, applied to
\eqref{eq:one-parameter-scale}, gives
\begin{equation}\label{eq:sequence_Version}
    \left\|
        \left(
            \sum_\nu
                \|\mathcal B_ag_\nu\|_{\ell^2}^2
        \right)^{1/2}
    \right\|_1
    \leq
    C\left\|
        \left(
            \sum_\nu
                \|\mathcal B_rg_\nu\|_{\ell^2}^2
        \right)^{1/2}
    \right\|_1.
\end{equation}
The second inequality in
\eqref{eq:one-parameter-scale} and
\eqref{eq:one-parameter-cutoff} similarly give
\begin{equation}\label{eq:sequence-scale-cutoff}
\begin{split}
    \left\|
        \left(
            \sum_\nu
                \|\mathcal B_rg_\nu\|_{\ell^2}^2
        \right)^{1/2}
    \right\|_1
    &\leq
    C\left(1+\log\frac{r}{a}\right)
    \left\|
        \left(
            \sum_\nu
                \|\mathcal B_ag_\nu\|_{\ell^2}^2
        \right)^{1/2}
    \right\|_1,\\
    \left\|
        \left(
            \sum_\nu
                \|\mathcal B_r(\chi g_\nu)\|_{\ell^2}^2
        \right)^{1/2}
    \right\|_1
    &\leq
    C\bigl(\|\chi\|_\infty+r\|\nabla\chi\|_\infty\bigr)
    \left\|
        \left(
            \sum_\nu
                \|\mathcal B_rg_\nu\|_{\ell^2}^2
        \right)^{1/2}
    \right\|_1.
\end{split}
\end{equation}

We first work with $f\in\mathcal S$.
For fixed $x_2$, consider the functions of $x_1$
\[
    g_k(x_1)
    =
    \bigl(\psi^{(2)}_{k,a_2}*_2f\bigr)(x_1,x_2),
    \qquad k\geq0.
\]
For every $u>0$,
\begin{align*}
    \left(
        \sum_{k\geq0}
            \|\mathcal B_{1,u}g_k(x_1)\|_{\ell^2}^2
    \right)^{1/2}
    &=
    \left(
        \sum_{j,k\geq0}
        \left|
            \psi^{(1)}_{j,u}*_1
            \psi^{(2)}_{k,a_2}*_2f(x_1,x_2)
        \right|^2
    \right)^{1/2}=S^{\mathrm{loc}}_{u,a_2}f(x_1,x_2).
\end{align*}
Thus \eqref{eq:sequence_Version} in $x_1$ gives
\[
    \int_{\mathbb R^{n_1}}
        S^{\mathrm{loc}}_{a_1,a_2}f(x_1,x_2)\,dx_1
    \leq
    C\int_{\mathbb R^{n_1}}
        S^{\mathrm{loc}}_{r_1,a_2}f(x_1,x_2)\,dx_1.
\]
Integrating in $x_2$ yields
\begin{equation}\label{eq:change-first-local-scale}
    \|f\|_{h^1_{a_1,a_2}}
    \leq C\|f\|_{h^1_{r_1,a_2}}.
\end{equation}

Next, for fixed $x_1$, apply
\eqref{eq:sequence_Version} in $x_2$ to
$
    \left(
        \bigl(\psi^{(1)}_{j,r_1}*_1f\bigr)(x_1,\cdot)
    \right)_{j\geq0}
$.
Convolutions in different variables commute,
so integration in $x_1$ gives
$
    \|f\|_{h^1_{r_1,a_2}}
    \leq C\|f\|_{h^1_{r_1,r_2}}
$.
Together with \eqref{eq:change-first-local-scale},
this proves \eqref{eq:local-scale-comparison}.

For the reverse comparison, use the first
inequality in \eqref{eq:sequence-scale-cutoff}.
Apply it first in $x_1$ to
$
    \left(
        \bigl(\psi^{(2)}_{k,r_2}*_2f\bigr)(\cdot,x_2)
    \right)_{k\geq0}
$,
and then in $x_2$ to
$
    \left(
        \bigl(\psi^{(1)}_{j,a_1}*_1f\bigr)(x_1,\cdot)
    \right)_{j\geq0}
$.
After integration in the other variable,
\begin{align*}
    \|f\|_{h^1_{r_1,r_2}}
    &\leq
    C\left(1+\log\frac{r_1}{a_1}\right)
        \|f\|_{h^1_{a_1,r_2}}\leq
    C\left(1+\log\frac{r_1}{a_1}\right)
     \left(1+\log\frac{r_2}{a_2}\right)
        \|f\|_{h^1_{a_1,a_2}}.
\end{align*}
This proves \eqref{eq:local-scale-comparison-reverse}.

\medskip
\noindent
\textbf{Step 4: The product cutoff estimate.}

For fixed $x_2$, set
\[
    g_k(x_1)
    =
    \bigl(\psi^{(2)}_{k,r_2}*_2f\bigr)(x_1,x_2).
\]
Since $\chi_1$ depends only on $x_1$,
$
    \psi^{(2)}_{k,r_2}*_2(\chi_1f)
    =
    \chi_1\bigl(\psi^{(2)}_{k,r_2}*_2f\bigr)
$.
Consequently, the second inequality in
\eqref{eq:sequence-scale-cutoff} gives
\begin{align*}
    &\int_{\mathbb R^{n_1}}
        S^{\mathrm{loc}}_{r_1,r_2}(\chi_1f)(x_1,x_2)
        \,dx_1=
    \int_{\mathbb R^{n_1}}
        \left(
            \sum_{k\geq0}
                \|\mathcal B_{1,r_1}
                    (\chi_1g_k)(x_1)\|_{\ell^2}^2
        \right)^{1/2}dx_1\\
    &\quad\leq
    C\bigl(\|\chi_1\|_\infty
           +r_1\|\nabla\chi_1\|_\infty\bigr)
    \int_{\mathbb R^{n_1}}
        S^{\mathrm{loc}}_{r_1,r_2}f(x_1,x_2)\,dx_1.
\end{align*}
After integration in $x_2$,
\[
    \|\chi_1f\|_{h^1_{r_1,r_2}}
    \leq
    C\bigl(\|\chi_1\|_\infty
           +r_1\|\nabla\chi_1\|_\infty\bigr)
    \|f\|_{h^1_{r_1,r_2}}.
\]
Interchanging the variables gives the analogous
estimate for multiplication by $\chi_2$.
Applying the two estimates successively proves
\eqref{eq:local-product-cutoff} for Schwartz functions.

All applications of the sequence estimates
may first be made to finite families.
Monotone convergence then gives the full
square sums.

\medskip
\noindent
\textbf{Step 5: Passage to general functions.}

Choose Schwartz approximants converging to $f$
in the norm on the right-hand side of the
estimate under consideration.
By \eqref{eq:product-local-L1-embedding},
they also converge to $f$ in $L^1$.
For the cutoff estimate, multiplication by
$\chi_1\chi_2$ preserves this $L^1$ convergence.

Lemma~\ref{lem:local-Hardy-Fatou} bounds the
norm on the left-hand side by the lower limit
of the corresponding norms of the approximants.
The estimates already proved for Schwartz
functions therefore pass to the limit.
\end{proof}

In particular, for fixed $\eta_1,\eta_2\in(0,1]$,
the norms of $h^1_{r_1,r_2}$ and
$h^1_{\eta_1r_1,\eta_2r_2}$ are equivalent,
uniformly in $r_1,r_2$.

\subsection{Local product atoms and support control}
\label{subsec:localized-atomic-decompositions}

We use the fixed parameters $\tau>1$, $\rho>0$,
and the families $\mathcal M_{r_1,r_2}(\Omega)$
introduced in Subsection~\ref{subsec:main-results}.
The following local atoms provide the model
for Definition~\ref{def:covering-product-atom}.

\begin{definition}\label{def:local-product-atom}
Let $r_1,r_2>0$.
A function $a\in L^2$ is a local product atom
at scales $r_1,r_2$ if there is an open set
$\Omega$ with $0<|\Omega|<\infty$ such that
the following conditions hold.

\begin{enumerate}
\item
The function $a$ vanishes almost everywhere
outside $\Omega$, and
\[
    \|a\|_2\leq|\Omega|^{-1/2}.
\]

\item
There are functions $a_R\in L^2$, indexed by
$R\in\mathcal M_{r_1,r_2}(\Omega)$, satisfying
\[
    a=\sum_{R\in\mathcal M_{r_1,r_2}(\Omega)}a_R
    \quad\text{in }L^2,
    \qquad
    \sum_{R\in\mathcal M_{r_1,r_2}(\Omega)}
        \|a_R\|_2^2
    \leq|\Omega|^{-1}.
\]
Each $a_R$ vanishes almost everywhere outside
$\tau R$.

\item
For $R=I\times J$, we have
\[
    \int_{\mathbb R^{n_1}}a_R(x_1,x_2)\,dx_1=0
    \quad\text{for almost every }x_2
    \quad\text{if }\ell(I)<\rho r_1,
\]
and
\[
    \int_{\mathbb R^{n_2}}a_R(x_1,x_2)\,dx_2=0
    \quad\text{for almost every }x_1
    \quad\text{if }\ell(J)<\rho r_2.
\]
No cancellation in a variable is required
when the corresponding side length is  $\geq \rho r_i$.
\end{enumerate}
\end{definition}

The components $a_R$ need not be supported
in $\Omega$.
Every atom from
Definition~\ref{def:covering-product-atom}
is a local product atom at scales
$\eta\ell(Q_1),\eta\ell(Q_2)$, with the
additional support restrictions specified there.

We use the local atomic estimates of
\cite[Theorems~1.4 and~1.5]{DingLuZhu2019}.
The constants in these estimates may depend
on the fixed support and cancellation parameters.
The normalization of the squared $L^2$ bounds
is obtained by a fixed rescaling of atoms
and coefficients.

\begin{proposition}\label{prop:local-product-atoms}
For every local product atom $a$ at scales
$r_1,r_2$,
\begin{equation}\label{eq:local-product-atom-bound}
    \|a\|_{h^1_{r_1,r_2}}\leq C.
\end{equation}
Moreover, every $f\in h^1_{r_1,r_2}$ admits
a representation
\begin{equation}\label{eq:local-product-atomic-decomposition}
    f=\sum_{\nu=1}^{\infty}\lambda_\nu a_\nu,
    \qquad
    \sum_{\nu=1}^{\infty}|\lambda_\nu|
    \leq C\|f\|_{h^1_{r_1,r_2}},
\end{equation}
where the $a_\nu$ are local product atoms
at scales $r_1,r_2$.
The series converges in $h^1_{r_1,r_2}$ and
in $L^1$.
If $f\in L^2$, it can also be chosen to
converge in $L^2$.
\end{proposition}

\begin{proof}
We explain the passage from unit scales to
$r_1,r_2>0$.
Use the dilation from
\eqref{eq:local-product-dilation}.
If $a$ is an atom associated with $\Omega$,
put
\[
    \widetilde\Omega
    =
    \{(u_1,u_2):(r_1u_1,r_2u_2)\in\Omega\}.
\]
Then
\[
    |\widetilde\Omega|
    =
    r_1^{-n_1}r_2^{-n_2}|\Omega|,
    \qquad
    \|\widetilde a\|_2
    =
    r_1^{n_1/2}r_2^{n_2/2}\|a\|_2
    \leq|\widetilde\Omega|^{-1/2}.
\]
The map
\[
    I\times J
    \longmapsto
    (r_1^{-1}I)\times(r_2^{-1}J)
\]
takes $\mathcal M_{r_1,r_2}(\Omega)$ bijectively
onto $\mathcal M_{1,1}(\widetilde\Omega)$.
The corresponding components are
\[
    \widetilde a_{\widetilde R}(u_1,u_2)
    =
    r_1^{n_1}r_2^{n_2}
        a_R(r_1u_1,r_2u_2).
\]
They vanish outside $\tau\widetilde R$ and satisfy
\[
    \sum_{\widetilde R}
        \|\widetilde a_{\widetilde R}\|_2^2
    =
    r_1^{n_1}r_2^{n_2}\sum_R\|a_R\|_2^2
    \leq|\widetilde\Omega|^{-1}.
\]
The rectangular series remains $L^2$-convergent.
Moreover,
$
    \ell(I)<\rho r_1
    \quad\Longleftrightarrow\quad
    \ell(r_1^{-1}I)<\rho
$,
and the analogous identity holds in the second
variable.
Changing variables in the partial integrals
therefore preserves the cancellation conditions.
Thus $\widetilde a$ is a unit-scale atom.

The unit-scale atomic bound and
\eqref{eq:local-product-dilation} give
$
    \|a\|_{h^1_{r_1,r_2}}
    =
    \|\widetilde a\|_{h^1_{1,1}}
    \leq C
$.
Similarly, apply the unit-scale atomic
decomposition to $\widetilde f$ and then
use the dilation.
The coefficients are unchanged, and the
atoms become atoms at scales $r_1,r_2$.
This proves
\eqref{eq:local-product-atomic-decomposition}.

The uniform atom bound and summability of the
coefficients give convergence in the Hardy norm.
Equation~\eqref{eq:product-local-L1-embedding}
then gives convergence in $L^1$.
The $L^2$ convergence supplied by the unit-scale
construction is also preserved by dilation.
\end{proof}

For the application to the covering cubes,
we need a decomposition in which the associated
open sets and every rectangular component remain
in a fixed enlargement of the original support.

\begin{lemma}\label{lem:localized-atomic-decomposition}
There exist constants $A\geq1$ and $C>0$
with the following property.

Let $P_i\subset\mathbb R^{n_i}$ be cubes with
$\ell(P_i)=r_i$, $i=1,2$, and let
$g\in h^1_{r_1,r_2}$ vanish almost everywhere
outside $P_1\times P_2$.
Then
\[
    g=\sum_{\nu=1}^{\infty}\lambda_\nu a_\nu
\]
in $h^1_{r_1,r_2}$ and in $L^1$, where each
$a_\nu$ is a local product atom at scales
$r_1,r_2$, and
\[
    \sum_{\nu=1}^{\infty}|\lambda_\nu|
    \leq C\|g\|_{h^1_{r_1,r_2}}.
\]
Their associated open sets satisfy
\[
    \Omega_\nu\subset AP_1\times AP_2,
\]
and every rectangular component $(a_\nu)_R$
vanishes almost everywhere outside
$AP_1\times AP_2$.

If $g\in L^2$, the series can also be chosen
to converge in $L^2$.
The constants are independent of $r_1,r_2$,
the positions of the cubes, and $g$.
\end{lemma}

\begin{proof}
We essentially follow the proof of
\cite[Theorem~1.5]{DingLuZhu2019}, keeping additional
control of the scales and supports throughout
the construction.

\medskip
\noindent
\textbf{Step 1: Unit scales and the reproducing formula.} By \eqref{eq:local-product-dilation}, it suffices
to consider $r_1=r_2=1$.
We first assume $g\in L^2\cap h^1_{1,1}$ and
prove the slightly stronger assertion allowing
\[
    \operatorname{supp}g\subset2P_1\times2P_2,
    \qquad \ell(P_1)=\ell(P_2)=1.
\]
This larger support will be useful in Step~10.

For $i=1,2$, choose
$\varphi^{(i)}_0\in C^\infty_c(\mathbb R^{n_i})$
with $\int\varphi^{(i)}_0=1$, and set
\[
    \varphi^{(i)}(x_i)
    =
    \varphi^{(i)}_0(x_i)
    -
    2^{-n_i}\varphi^{(i)}_0(x_i/2).
\]
Then $\int\varphi^{(i)}=0$.
By \cite[Theorem~1.1]{DingLuZhu2019}, there are
$\eta^{(i)}_0,\eta^{(i)}\in C^\infty_c(\mathbb R^{n_i})$,
with $\int\eta^{(i)}=0$, such that, writing
\[
    \varphi^{(i)}_j(x_i)
    =
    2^{jn_i}\varphi^{(i)}(2^jx_i),
    \qquad
    \eta^{(i)}_j(x_i)
    =
    2^{jn_i}\eta^{(i)}(2^jx_i),
    \qquad j\geq1,
\]
we have
\[
    u=\sum_{j\geq0}
        \eta^{(i)}_j*(\varphi^{(i)}_j*u)
    \quad\text{in }L^2.
\]
Set $\varphi_{j,k}
    =
    \varphi^{(1)}_j\otimes\varphi^{(2)}_k$ and $\eta_{j,k}
    =
    \eta^{(1)}_j\otimes\eta^{(2)}_k$. Applying the formula in both variables gives
\begin{equation}\label{eq:localized-reproducing-formula}
    g=\sum_{j,k\geq0}
        \eta_{j,k}*(\varphi_{j,k}*g)
    \quad\text{in }L^2.
\end{equation}
There is a fixed $K>0$ such that
\[
    \operatorname{supp}\varphi^{(i)}_j
    \cup
    \operatorname{supp}\eta^{(i)}_j
    \subset[-K2^{-j},K2^{-j}]^{n_i},
    \qquad j\geq0,\quad i=1,2.
\]

\medskip
\noindent
\textbf{Step 2: Choice of scales and rectangular pieces.} Recall that $\rho>0$ is the fixed constant
in Definition~\ref{def:local-product-atom}
determining the cancellation thresholds:
at scales $r_1,r_2$, cancellation is required
when $\ell(I)<\rho r_1$ or $\ell(J)<\rho r_2$,
respectively.
Fix a dyadic number $b\geq1$ such that
\begin{equation}\label{eq:localized-choice-b}
    b\geq\rho,
    \qquad
    1+\frac{2K}{b}<\tau.
\end{equation}
For $j,k\geq0$, let $\mathcal Q_{j,k}$ be the
partition of $\mathbb R^{n_1}\times\mathbb R^{n_2}$
into dyadic rectangles $Q=I\times J$, where
$I\in\mathcal D_1$ and $J\in\mathcal D_2$ have
side lengths
\[
    \ell(I)=b2^{-j},
    \qquad
    \ell(J)=b2^{-k}.
\]
For $Q\in\mathcal Q_{j,k}$, define
\begin{equation}\label{eq:u_Q}
\begin{split}
    c_Q
    &=
    \operatorname*{ess\,sup}_{y\in Q}
        |(\varphi_{j,k}*g)(y)|,\ \ \
    u_Q(x)
    =
    \int_Q
        \eta_{j,k}(x-y)(\varphi_{j,k}*g)(y)\,dy.
\end{split}
\end{equation}
Then
\begin{equation}\label{eq:localized-piece-support}
    \operatorname{supp}u_Q
    \subset
    \left(1+\frac{2K}{b}\right)Q
    \subset\tau Q.
\end{equation}
Define
\[
    S_bg(x)
    =
    \big(
        \sum_{j,k\geq0}
        \sum_{Q\in\mathcal Q_{j,k}}
            c_Q^2\mathbf1_Q(x)
    \big)^{1/2}.
\]

\medskip
\noindent
\textbf{Step 3: The discrete square-function estimate.} We claim that
\begin{equation}\label{eq:localized-discrete-square}
    \|S_bg\|_1\leq C_b\|g\|_{h^1_{1,1}}.
\end{equation}
For $b=1$, this is
\cite[Corollary~3.7]{DingLuZhu2019}.
For the chosen $b$, set
\[
    F_{j,k}(x)
    =
    \sum_{\substack{I,J\ \mathrm{dyadic}\\
                    \ell(I)=2^{-j},\ \ell(J)=2^{-k}}}
    \left(
        \operatorname*{ess\,sup}_{y\in I\times J}
            |(\varphi_{j,k}*g)(y)|
    \right)\mathbf1_{I\times J}(x).
\]
Each $Q\in\mathcal Q_{j,k}$ is the union of
$b^{n_1+n_2}$ rectangles from this finer partition, so, for almost every $x\in Q$,
\[
    c_Q^{1/2}
    \leq
    b^{n_1+n_2}
    \frac1{|Q|}\int_QF_{j,k}(y)^{1/2}\,dy
    \leq
    b^{n_1+n_2}
    M_{\mathrm{str}}(F_{j,k}^{1/2})(x).
\]
Here
\[
    M_{\mathrm{str}}u(x_1,x_2)
    =
    \sup_{\substack{I\ni x_1\\J\ni x_2}}
    \frac1{|I|\,|J|}
        \int_I\int_J|u(y_1,y_2)|\,dy_2\,dy_1,
\]
where $I,J$ range over axis-parallel cubes
with independently chosen side lengths.

The Fefferman--Stein inequality
\cite{FeffermanStein}, with $p=2$ and $q=4$,
gives
\begin{align*}
    \|S_bg\|_1
    &\leq
    b^{2(n_1+n_2)}
    \left\|
        \left(
            \sum_{j,k\geq0}
                |M_{\mathrm{str}}(F_{j,k}^{1/2})|^4
        \right)^{1/4}
    \right\|_2^2\\
    &\leq
    C_b\left\|
        \left(
            \sum_{j,k\geq0}F_{j,k}^2
        \right)^{1/4}
    \right\|_2^2=
    C_b\left\|
        \left(
            \sum_{j,k\geq0}F_{j,k}^2
        \right)^{1/2}
    \right\|_1
    \leq C_b\|g\|_{h^1_{1,1}}.
\end{align*}
To justify its use for $M_{\mathrm{str}}$,
let $M_i$ be the Hardy--Littlewood maximal
operator in $x_i$.
Since $M_{\mathrm{str}}u\leq M_1(M_2u)$,
two one-parameter applications give
\begin{align*}
    \left\|
        \left(
            \sum_{j,k}|M_{\mathrm{str}}u_{j,k}|^4
        \right)^{1/4}
    \right\|_2
    &\leq
    \left\|
        \left(
            \sum_{j,k}|M_1(M_2u_{j,k})|^4
        \right)^{1/4}
    \right\|_2\\
    &\leq
    C\left\|
        \left(
            \sum_{j,k}|M_2u_{j,k}|^4
        \right)^{1/4}
    \right\|_2\leq
    C\left\|
        \left(
            \sum_{j,k}|u_{j,k}|^4
        \right)^{1/4}
    \right\|_2.
\end{align*}
We apply this with $u_{j,k}=F_{j,k}^{1/2}$.
Infinite families follow by monotone convergence.

\medskip
\noindent
\textbf{Step 4: An $L^2$ estimate for sums of pieces.} For every finite collection $\mathcal E$
of coefficient rectangles,
\begin{equation}\label{eq:localized-synthesis-estimate}
    \left\|\sum_{Q\in\mathcal E}u_Q\right\|_2^2
    \leq
    C\sum_{j,k\geq0}
    \sum_{Q\in\mathcal E\cap\mathcal Q_{j,k}}
        \int_Q|\varphi_{j,k}*g|^2.
\end{equation}
Indeed, write
$\eta_{j,k}^*(x)=\overline{\eta_{j,k}(-x)}$.
For $h\in L^2$, Cauchy--Schwarz gives
\begin{align*}
    \left|
        \left\langle\sum_{Q\in\mathcal E}u_Q,h\right\rangle
    \right|
    &\leq
    \left(
        \sum_{j,k\geq0}
        \sum_{Q\in\mathcal E\cap\mathcal Q_{j,k}}
            \int_Q|\varphi_{j,k}*g|^2
    \right)^{1/2}\times
    \left(
        \sum_{j,k\geq0}
        \sum_{Q\in\mathcal E\cap\mathcal Q_{j,k}}
            \int_Q|\eta_{j,k}^**h|^2
    \right)^{1/2}.
\end{align*}
For fixed $j,k$, the rectangles have disjoint
interiors. Hence the second factor is bounded by
\[
    \big(
        \sum_{j,k\geq0}\|\eta_{j,k}^**h\|_2^2
    \big)^{1/2}
    \leq C\|h\|_2.
\]
The last inequality follows from Plancherel's
theorem and
\[
    \sup_{\xi_i}
        \sum_{j\geq0}
            |\widehat{\eta^{(i)}_j}(\xi_i)|^2
    <\infty.
\]
The low-frequency term is bounded, while
cancellation and smoothness give
\[
    |\widehat{\eta^{(i)}_j}(\xi_i)|
    \leq
    C\min\{2^{-j}|\xi_i|,
                  (2^{-j}|\xi_i|)^{-1}\},
    \qquad j\geq1,\quad \xi_i\neq0.
\]
Splitting the sum at $2^{-j}|\xi_i|=1$
proves the uniform bound.
Duality now yields
\eqref{eq:localized-synthesis-estimate}.

The kernels $\varphi^{(i)}_j$ satisfy the same
Fourier estimates used above for $\eta^{(i)}_j$.
In particular,
\[
    \sup_{\xi_i\in\mathbb R^{n_i}}
        \sum_{j\geq0}
            |\widehat{\varphi^{(i)}_j}(\xi_i)|^2
    <\infty.
\]
Since
$\varphi_{j,k}=\varphi^{(1)}_j\otimes\varphi^{(2)}_k$,
Plancherel's theorem and Tonelli's theorem give
\begin{align*}
    \sum_{j,k\geq0}\|\varphi_{j,k}*g\|_2^2
    &=
    \int\limits_{\mathbb R^{n_1}\times\mathbb R^{n_2}}
        |\widehat g(\xi_1,\xi_2)|^2
        \left(
            \sum_{j\geq0}
                |\widehat{\varphi^{(1)}_j}(\xi_1)|^2
        \right)\times
        \left(
            \sum_{k\geq0}
                |\widehat{\varphi^{(2)}_k}(\xi_2)|^2
        \right)d\xi_1\,d\xi_2\leq C\|g\|_2^2.
\end{align*}
Moreover, for each fixed $j,k$,
$\mathcal Q_{j,k}$ partitions the whole product
space. Hence
\begin{equation}\label{eq:localized-total-energy}
    \sum_{j,k\geq0}
    \sum_{Q\in\mathcal Q_{j,k}}
        \int_Q|\varphi_{j,k}*g|^2
    =
    \sum_{j,k\geq0}\|\varphi_{j,k}*g\|_2^2
    \leq C\|g\|_2^2.
\end{equation}

We now prove
convergence of $\sum_Q u_Q$.
Given $\varepsilon>0$, choose a finite collection
$\mathcal E_0$ of coefficient rectangles such that
\[
    \sum_{j,k\geq0}
    \sum_{Q\in\mathcal Q_{j,k}\setminus\mathcal E_0}
        \int_Q|\varphi_{j,k}*g|^2
    <\varepsilon^2.
\]
For every finite collection $\mathcal E$
disjoint from $\mathcal E_0$,
\eqref{eq:localized-synthesis-estimate} gives
\[
    \left\|\sum_{Q\in\mathcal E}u_Q\right\|_2^2
    \leq
    C\sum_{j,k\geq0}
    \sum_{Q\in\mathcal E\cap\mathcal Q_{j,k}}
        \int_Q|\varphi_{j,k}*g|^2
    \leq C\varepsilon^2.
\]
Thus every finite sum of terms outside
$\mathcal E_0$ has small $L^2$ norm.
This is the unconditional Cauchy criterion:
the series converges in $L^2$, independently
of the order of its terms.

The same argument applies to any subcollection.
Passing to the limit in
\eqref{eq:localized-synthesis-estimate}
therefore extends that estimate to infinite
collections, with the sum on the left
interpreted as an $L^2$ limit.

It remains to identify the sum of the full series.
For fixed $j,k$, the partition property gives
\[
    \sum_{Q\in\mathcal Q_{j,k}}
        \mathbf1_Q(\varphi_{j,k}*g)
    =
    \varphi_{j,k}*g
    \quad\text{in }L^2.
\]
Convolution with $\eta_{j,k}$ is bounded on
$L^2$, since $\eta_{j,k}\in L^1$.
Using \eqref{eq:u_Q}, we consequently obtain
\[
    \sum_{Q\in\mathcal Q_{j,k}}u_Q
    =
    \eta_{j,k}*(\varphi_{j,k}*g)
    \quad\text{in }L^2.
\]
Unconditional convergence allows us to group
the terms according to $j,k$.
Equation~\eqref{eq:localized-reproducing-formula}
then yields
\[
    \sum_Q u_Q
    =
    \sum_{j,k\geq0}
        \eta_{j,k}*(\varphi_{j,k}*g)
    =
    g
    \quad\text{in }L^2.
\]

\medskip
\noindent
\textbf{Step 5: Grouping by level sets.} For $m\in\mathbb Z$, put
\[
    E_m=\{x:S_bg(x)>2^m\},
\]
and
\[
    \mathcal B_m
    =
    \left\{
        Q:
        |Q\cap E_m|>\frac{|Q|}{2},
        \quad
        |Q\cap E_{m+1}|\leq\frac{|Q|}{2}
    \right\}.
\]
Every rectangle with $c_Q>0$ belongs to exactly
one $\mathcal B_m$.
Indeed, $S_bg\geq c_Q$ almost everywhere on $Q$,
so $|Q\cap E_m|=|Q|$ for sufficiently negative $m$.
On the other hand,
\[
    |E_m|\leq2^{-m}\|S_bg\|_1\longrightarrow0
    \qquad(m\to\infty).
\]
The required index is the largest $m$ with
$|Q\cap E_m|>|Q|/2$. Fix
\[
    0<\theta<\frac1{2\tau^{n_1+n_2}},
    \qquad
    \Omega_m
    =
    \{M_{\mathrm{str}}\mathbf1_{E_m}>\theta\}.
\]
The sets $\Omega_m$ are open, and
\begin{equation}\label{eq:localized-level-measure}
    |\Omega_m|
    \leq
    \theta^{-2}
        \|M_{\mathrm{str}}\mathbf1_{E_m}\|_2^2
    \leq C_\theta|E_m|.
\end{equation}
For $Q\in\mathcal B_m$ and $x\in\tau Q$,
\[
    M_{\mathrm{str}}\mathbf1_{E_m}(x)
    \geq\frac{|Q\cap E_m|}{|\tau Q|}
    >
    \frac1{2\tau^{n_1+n_2}}
    >\theta.
\]
Therefore,
\begin{equation}\label{eq:localized-halo-inclusion}
    \tau Q\subset\Omega_m,
    \qquad Q\in\mathcal B_m.
\end{equation}

\medskip
\noindent
\textbf{Step 6: Definition and normalization of the atoms.} We have
\begin{align}
    &\sum_{j,k\geq0}
    \sum_{Q\in\mathcal B_m\cap\mathcal Q_{j,k}}
        \int_Q|\varphi_{j,k}*g|^2
    \leq
    \sum_{Q\in\mathcal B_m}c_Q^2|Q|
    \leq
    2\sum_{Q\in\mathcal B_m}
        c_Q^2|Q\setminus E_{m+1}|
    \leq
    2\int_{\Omega_m\setminus E_{m+1}}(S_bg)^2
    \leq2^{2m+3}|\Omega_m|.
    \label{eq:localized-level-energy}
\end{align}
The second inequality uses
$|Q\setminus E_{m+1}|\geq|Q|/2$.
For the third, Tonelli's theorem and
$Q\subset\Omega_m$ give
\[
    \sum_{Q\in\mathcal B_m}
        c_Q^2|Q\setminus E_{m+1}|
    =
    \int_{\Omega_m\setminus E_{m+1}}
        \sum_{Q\in\mathcal B_m}
            c_Q^2\mathbf1_Q(x)\,dx.
\]
The sum inside the integral is bounded by
$(S_bg(x))^2$, since it contains only some
of its nonnegative terms.
No disjointness of these rectangles is needed.
Finally, $S_bg\leq2^{m+1}$ outside $E_{m+1}$.

For each nonempty group, define
\[
    \lambda_m=C_0\,2^m|\Omega_m|,
    \qquad
    a_m=\lambda_m^{-1}
        \sum_{Q\in\mathcal B_m}u_Q.
\]
The sum converges in $L^2$ by Step~4.
Equations \eqref{eq:localized-piece-support}
and \eqref{eq:localized-halo-inclusion} imply
that $a_m$ vanishes outside $\Omega_m$.
Moreover,
\[
    \|a_m\|_2^2
    \leq
    \frac{C\,2^{2m}|\Omega_m|}{\lambda_m^2}
    \leq|\Omega_m|^{-1}
\]
for a sufficiently large fixed $C_0$.

\medskip
\noindent
\textbf{Step 7: Additional control of the open sets.} We now prove the additional support estimate
needed in this lemma: every $\Omega_m$ is
contained in one fixed enlargement of
$P_1\times P_2$.
This control must be retained throughout the
atomic construction.

Let $Q=I\times J\in\mathcal Q_{j,k}$ with
$c_Q>0$.
There is $y=(y_1,y_2)\in Q$ such that
$(\varphi_{j,k}*g)(y)\neq0$.
Since $g$ is supported in $2P_1\times2P_2$,
\[
    y_1\in2P_1+[-K2^{-j},K2^{-j}]^{n_1},
    \qquad
    y_2\in2P_2+[-K2^{-k},K2^{-k}]^{n_2}.
\]
Here $+$ denotes the sum of sets.
If $p_i$ is the center of $P_i$, then
\[
    |y_i-p_i|_\infty\leq1+K,
\]
because $\ell(P_i)=1$ and $j,k\geq0$.
The side lengths of $Q$ are at most $b$,
so every $x\in Q$ satisfies
\[
    |x_i-p_i|_\infty
    \leq|x_i-y_i|_\infty+|y_i-p_i|_\infty
    \leq b+1+K.
\]
For a fixed $c>2(b+1+K)$, we therefore have
\[
    \bigcup_{\{Q:c_Q>0\}}Q
    \subset cP_1\times cP_2.
\]
Consequently,
\[
    E_m\subset cP_1\times cP_2
    \qquad\text{for every }m.
\]

For arbitrary cubes $I\ni x_1$ and $J\ni x_2$,
\[
    \frac1{|I|\,|J|}
        \int_{I\times J}\mathbf1_{E_m}
    \leq
    \frac{|I\cap cP_1|}{|I|}
    \frac{|J\cap cP_2|}{|J|}.
\]
Taking suprema gives
\[
    M_{\mathrm{str}}\mathbf1_{E_m}(x_1,x_2)
    \leq
    M_1\mathbf1_{cP_1}(x_1)
    M_2\mathbf1_{cP_2}(x_2).
\]
Both factors lie in $[0,1]$.
Their product can exceed $\theta$ only if
both factors exceed $\theta$, and hence
\[
    \Omega_m
    \subset
    \{M_1\mathbf1_{cP_1}>\theta\}
    \times
    \{M_2\mathbf1_{cP_2}>\theta\}.
\]

For a cube $P\subset\mathbb R^n$ and $x\notin P$,
every cube $I$ containing $x$ and intersecting
$P$ satisfies
\[
    \sqrt n\,\ell(I)\geq\operatorname{dist}(x,P).
\]
Thus
\[
    \frac{|I\cap P|}{|I|}
    \leq
    n^{n/2}
    \left(
        \frac{\ell(P)}{\operatorname{dist}(x,P)}
    \right)^n.
\]
Cubes disjoint from $P$ contribute zero, so
\[
    M\mathbf1_P(x)
    \leq
    n^{n/2}
    \left(
        \frac{\ell(P)}{\operatorname{dist}(x,P)}
    \right)^n.
\]
It follows that $\{M\mathbf1_P>\theta\}$
lies in a fixed concentric enlargement of $P$.
Applying this to $cP_1,cP_2$ gives
\begin{equation}\label{eq:localized-open-set-support}
    \Omega_m\subset A_0P_1\times A_0P_2
    \qquad\text{for every }m,
\end{equation}
with $A_0$ independent of $g$, $m$, and the
positions of the cubes.

\medskip
\noindent
\textbf{Step 8: Rectangular components and cancellation.} Every $Q\in\mathcal B_m$ is contained in a
rectangle from $\mathcal M_{1,1}(\Omega_m)$.
Indeed, $Q\subset\Omega_m$, and boundedness
of $\Omega_m$ leaves only finitely many dyadic
rectangles containing $Q$ and contained in
$\Omega_m$.
Choose one maximal such rectangle $R(Q)$,
assigning each $Q$ only once.

For $R\in\mathcal M_{1,1}(\Omega_m)$, define
\[
    (a_m)_R
    =
    \lambda_m^{-1}
    \sum_{{Q\in\mathcal B_m, \ R(Q)=R}}u_Q.
\]
By \eqref{eq:localized-synthesis-estimate},
\[
    \|(a_m)_R\|_2^2
    \leq
    \frac{C}{\lambda_m^2}
    \sum_{j,k\geq0}
    \sum_{\substack{Q\in\mathcal B_m\cap\mathcal Q_{j,k}\\
                    R(Q)=R}}
        \int_Q|\varphi_{j,k}*g|^2.
\]
For infinite collections, this follows by
passing to the $L^2$ limit of finite sums.
Since each $Q$ is assigned exactly once,
\eqref{eq:localized-level-energy} gives
\[
    \sum_{R\in\mathcal M_{1,1}(\Omega_m)}
        \|(a_m)_R\|_2^2
    \leq
    \frac{C\,2^{2m}|\Omega_m|}{\lambda_m^2}
    \leq|\Omega_m|^{-1},
\]
after increasing $C_0$ if necessary.
The same estimate for tails gives
\[
    a_m
    =
    \sum_{R\in\mathcal M_{1,1}(\Omega_m)}(a_m)_R
    \quad\text{in }L^2.
\]
Since $\tau Q\subset\tau R(Q)$,
each component is supported in $\tau R$.

Let $R=I_R\times J_R$ and suppose
$\ell(I_R)<\rho$.
Every assigned $Q\in\mathcal Q_{j,k}$ satisfies
\[
    b2^{-j}\leq\ell(I_R)<\rho\leq b.
\]
Therefore $j\geq1$, and $\int\eta^{(1)}_j=0$.
By \eqref{eq:u_Q} and Fubini's theorem,
\[
    \int_{\mathbb R^{n_1}}u_Q(x_1,x_2)\,dx_1=0.
\]
This cancellation passes to the $L^2$ sum,
because all terms are supported in $\tau R$
and
\[
    \left\|
        \int_{\tau I_R}v(x_1,\cdot)\,dx_1
    \right\|_{L^2(\mathbb R^{n_2})}
    \leq|\tau I_R|^{1/2}\|v\|_2.
\]
Hence
\[
    \int_{\mathbb R^{n_1}}(a_m)_R(x_1,x_2)\,dx_1=0
    \quad\text{for almost every }x_2.
\]

Similarly, $\ell(J_R)<\rho$ forces $k\geq1$
for every assigned rectangle.
Using $\int\eta^{(2)}_k=0$ and
\[
    \left\|
        \int_{\tau J_R}v(\cdot,x_2)\,dx_2
    \right\|_{L^2(\mathbb R^{n_1})}
    \leq|\tau J_R|^{1/2}\|v\|_2,
\]
we obtain cancellation in the second variable.
Thus $a_m$ satisfies
Definition~\ref{def:local-product-atom}.

Finally, \eqref{eq:localized-open-set-support}
implies
\[
    R\subset A_0P_1\times A_0P_2,
    \qquad
    \operatorname{supp}(a_m)_R
    \subset\tau R
    \subset(\tau A_0P_1)\times(\tau A_0P_2).
\]
Taking $A=\tau A_0$ gives the required common
support bound.

\medskip
\noindent
\textbf{Step 9: Coefficient bounds and convergence.} Equations \eqref{eq:localized-level-measure}
and \eqref{eq:localized-discrete-square} give
\begin{align*}
    \sum_m|\lambda_m|
    &=
    C_0\sum_m2^m|\Omega_m|
    \leq C\sum_m2^m|E_m|=
    C\int
        \sum_{\{m:2^m<S_bg(x)\}}2^m\,dx\ leq C\|S_bg\|_1
    \leq C\|g\|_{h^1_{1,1}}.
\end{align*}
The sums involving $\lambda_m$ run over the
nonempty groups.

Every nonzero $u_Q$ occurs in exactly one
$\lambda_ma_m$.
The unconditional $L^2$ convergence from Step~4
therefore yields
\[
    g=\sum_m\lambda_ma_m
    \quad\text{in }L^2.
\]
More explicitly, the squared $L^2$ norm of
the tail is bounded by
\[
    C\sum_{|m|>N}
    \sum_{j,k\geq0}
    \sum_{Q\in\mathcal B_m\cap\mathcal Q_{j,k}}
        \int_Q|\varphi_{j,k}*g|^2,
\]
which tends to zero by
\eqref{eq:localized-total-energy}.

By \eqref{eq:local-product-atom-bound}, the
series also converges absolutely in $h^1_{1,1}$,
and hence in $L^1$ by
\eqref{eq:product-local-L1-embedding}.
The $L^2$ and $L^1$ limits agree as distributions,
so both are equal to $g$.

\medskip
\noindent
\textbf{Step 10: Removal of the $L^2$ assumption.}

By the density of Schwartz functions in
$h^1_{1,1}$, recalled above, choose
$u_m\in\mathcal S$ such that
\[
    \|u_m-g\|_{h^1_{1,1}}
    \leq C_\chi^{-1}2^{-m}\|g\|_{h^1_{1,1}}.
\]
By \eqref{eq:product-local-L1-embedding},
this convergence also holds in $L^1$.
Set $g_m=\chi_1\chi_2u_m$.
Then $g_m\in C^\infty_c$ is supported in
$2P_1\times2P_2$ and, since $\chi_1\chi_2g=g$,
\[
    \|g_m-g\|_{h^1_{1,1}}
    =
    \|\chi_1\chi_2(u_m-g)\|_{h^1_{1,1}}
    \leq 2^{-m}\|g\|_{h^1_{1,1}}.
\]
and set $g_m=\chi_1\chi_2u_m$.
Then $g_m\in C^\infty_c$ is supported in
$2P_1\times2P_2$, and
\[
    \|g_m-g\|_{h^1_{1,1}}
    \leq2^{-m}\|g\|_{h^1_{1,1}},
\]
because $\chi_1\chi_2g=g$.

The telescoping identity gives
\[
    g=g_1+\sum_{m=2}^{\infty}(g_m-g_{m-1})
    \quad\text{in }h^1_{1,1}.
\]
Moreover,
$
    \|g_1\|_{h^1_{1,1}}
    \leq\frac32\|g\|_{h^1_{1,1}}
$,
and
\[
    \|g_m-g_{m-1}\|_{h^1_{1,1}}
    \leq
    (2^{-m}+2^{-(m-1)})\|g\|_{h^1_{1,1}}.
\]
Hence
\[
    \|g_1\|_{h^1_{1,1}}
    +
    \sum_{m=2}^{\infty}
        \|g_m-g_{m-1}\|_{h^1_{1,1}}
    \leq3\|g\|_{h^1_{1,1}}.
\]

Apply the $L^2$ case to $g_1$ and each
$g_m-g_{m-1}$.
All these functions are supported in
$2P_1\times2P_2$, so their decompositions have
open sets and rectangular components in the
same rectangle $AP_1\times AP_2$.
Writing the resulting coefficients as
$\lambda_{m,\nu}$, we obtain
\[
    \sum_{m,\nu}|\lambda_{m,\nu}|
    \leq C\|g\|_{h^1_{1,1}}.
\]
The uniform atom estimate gives absolute
convergence of the combined series in
$h^1_{1,1}$.
Summing first in $\nu$ recovers the telescoping
series, so its sum is $g$.
Enumerating all terms gives the required
decomposition.
Its convergence in $L^1$ follows from
\eqref{eq:product-local-L1-embedding}.

\medskip
\noindent
\textbf{Step 11: Return to arbitrary scales.}

Apply the unit-scale result to
\[
    \widetilde g(u_1,u_2)
    =
    r_1^{n_1}r_2^{n_2}g(r_1u_1,r_2u_2),
\]
supported in the unit cubes
$r_1^{-1}P_1\times r_2^{-1}P_2$.
If
$\widetilde g=\sum_\nu\lambda_\nu\widetilde a_\nu$,
put
\[
    a_\nu(x_1,x_2)
    =
    r_1^{-n_1}r_2^{-n_2}
    \widetilde a_\nu(x_1/r_1,x_2/r_2).
\]
The atomic scaling verified in
Proposition~\ref{prop:local-product-atoms}
and \eqref{eq:local-product-dilation} give
\[
    g=\sum_\nu\lambda_\nu a_\nu,
    \qquad
    \sum_\nu|\lambda_\nu|
    \leq C\|g\|_{h^1_{r_1,r_2}}.
\]
The open sets and rectangular components
are carried into $AP_1\times AP_2$.
Convergence in the Hardy norm, in $L^1$,
and, when applicable, in $L^2$, is preserved
by the dilation.
\end{proof}

\subsection{Atomic decompositions inside covering cubes}
\label{subsec:smaller-rectangles}

We now apply
Lemma~\ref{lem:localized-atomic-decomposition}
to functions supported in $Q_1^*\times Q_2^*$.
We first subdivide their supports into smaller
rectangles whose fixed enlargements remain
inside $Q_1^{**}\times Q_2^{**}$.

\begin{lemma}[Finite localization at smaller scales]
\label{lem:finite-localization}
Let $Q_i\subset\mathbb R^{n_i}$ be cubes and
$0<\eta_i\leq1/4$, $i=1,2$.
There are finite families of cubes $\mathcal P_i$
and smooth functions
$\{\chi_{i,P}\}_{P\in\mathcal P_i}$ such that
\[
    \ell(P)=\eta_i\ell(Q_i),\qquad
    P\cap Q_i^*\neq\varnothing,\qquad
    P\subset Q_i^{**},
\]
and
\begin{equation}\label{eq:small-cutoff-properties}
\begin{split}
    &0\leq\chi_{i,P}\leq1,\qquad
    \operatorname{supp}\chi_{i,P}\subset P,\\
    &\sum_{P\in\mathcal P_i}\chi_{i,P}=1
        \quad\text{on }Q_i^*,\\
    &\|\nabla\chi_{i,P}\|_\infty
        \leq\frac{C}{\eta_i\ell(Q_i)},\qquad
    \#\mathcal P_i\leq C\eta_i^{-n_i}.
\end{split}
\end{equation}

For every $g\in h^1_{\ell(Q_1),\ell(Q_2)}$
supported in $Q_1^*\times Q_2^*$,
\begin{equation}\label{eq:finite-localization-decomposition}
    g=
    \sum_{P_1\in\mathcal P_1}
    \sum_{P_2\in\mathcal P_2}
        \chi_{1,P_1}(x_1)\chi_{2,P_2}(x_2)g,
\end{equation}
and
\begin{equation}\label{eq:finite-localization-norm}
\begin{split}
    &\sum_{P_1\in\mathcal P_1}
    \sum_{P_2\in\mathcal P_2}
        \|\chi_{1,P_1}\chi_{2,P_2}g\|
            _{h^1_{\eta_1\ell(Q_1),\eta_2\ell(Q_2)}}\leq
    C\eta_1^{-n_1}\eta_2^{-n_2}
        \|g\|_{h^1_{\ell(Q_1),\ell(Q_2)}}.
\end{split}
\end{equation}
If $A\geq1$ and $\eta_i\leq1/(2A)$, then also
\begin{equation}\label{eq:small-cube-enlargement}
    AP\subset Q_i^{**},
    \qquad P\in\mathcal P_i.
\end{equation}
The constants are independent of the cubes,
$\eta_i$, and $g$.
\end{lemma}

\begin{proof}
For each $i$, take cubes of side length
$\eta_i\ell(Q_i)$ whose centers form a lattice
of spacing $\eta_i\ell(Q_i)/2$.
A smooth partition of unity subordinate to
these cubes is obtained by translating and
dilating a fixed cutoff and dividing by the
sum of its translates.
The bounded overlap and uniform lower bound
for this sum give
\[
    0\leq\chi_{i,P}\leq1,\qquad
    \operatorname{supp}\chi_{i,P}\subset P,\qquad
    \|\nabla\chi_{i,P}\|_\infty
    \leq\frac{C}{\eta_i\ell(Q_i)}.
\]
Retain only cubes intersecting $Q_i^*$.
Their cutoffs sum to one on $Q_i^*$,
and a lattice count gives
$\#\mathcal P_i\leq C\eta_i^{-n_i}$. Since $P\cap Q_i^*\neq\varnothing$, its center
differs from the center of $Q_i$ by at most
$
    \ell(Q_i)+\frac{\eta_i\ell(Q_i)}2
$
in each coordinate.
Every point of $P$ therefore differs from
that center by at most
$(1+\eta_i)\ell(Q_i)<2\ell(Q_i)$.
Thus $P\subset Q_i^{**}$.

The partition identities give
\eqref{eq:finite-localization-decomposition}.
By \eqref{eq:local-product-cutoff}, at scales
$\eta_i\ell(Q_i)$,
\[
    \|\chi_{1,P_1}\chi_{2,P_2}g\|
        _{h^1_{\eta_1\ell(Q_1),\eta_2\ell(Q_2)}}
    \leq
    C\|g\|_{h^1_{\eta_1\ell(Q_1),\eta_2\ell(Q_2)}}.
\]
Equation~\eqref{eq:local-scale-comparison}
bounds the right-hand side by
$C\|g\|_{h^1_{\ell(Q_1),\ell(Q_2)}}$.
Summing over the finite families proves
\eqref{eq:finite-localization-norm}.

Finally, every point of $AP$ differs from
the center of $Q_i$ by at most
\[
    \left(1+\frac{(A+1)\eta_i}{2}\right)\ell(Q_i)
    \leq
    \left(1+\frac{A+1}{4A}\right)\ell(Q_i)
    \leq\frac32\ell(Q_i)<2\ell(Q_i)
\]
in each coordinate.
This proves \eqref{eq:small-cube-enlargement}.
\end{proof}

\begin{corollary}\label{cor:atoms-inside-enlargements}
Let $A\geq1$ be the constant from
Lemma~\ref{lem:localized-atomic-decomposition},
and fix
\[
    0<\eta\leq
    \min\left\{\frac14,\frac1{2A}\right\}.
\]
Let $Q_i\subset\mathbb R^{n_i}$ be cubes, and
let $g\in h^1_{\ell(Q_1),\ell(Q_2)}$ vanish
almost everywhere outside $Q_1^*\times Q_2^*$.
Then
$
    g=\sum_{\nu=1}^{\infty}\lambda_\nu a_\nu
$
in $h^1_{\ell(Q_1),\ell(Q_2)}$ and in $L^1$,
where the $a_\nu$ are local product atoms
at scales $\eta\ell(Q_1),\eta\ell(Q_2)$ and
\[
    \sum_{\nu=1}^{\infty}|\lambda_\nu|
    \leq C\|g\|_{h^1_{\ell(Q_1),\ell(Q_2)}}.
\]
Their associated open sets and all rectangular
components are supported inside
$Q_1^{**}\times Q_2^{**}$.
Moreover,
\begin{equation}\label{eq:localized-atom-original-scales}
    \|a_\nu\|_{h^1_{\ell(Q_1),\ell(Q_2)}}\leq C.
\end{equation}
If $g\in L^2$, the decomposition can also
be chosen to converge in $L^2$.
The constants may depend on the fixed $\eta$,
but are independent of $Q_1,Q_2$ and $g$.

In particular, if $Q_i\in\mathcal Q_i$, every
atom in this decomposition satisfies
Definition~\ref{def:covering-product-atom}.
\end{corollary}

\begin{proof}
Apply Lemma~\ref{lem:finite-localization}
with $\eta_1=\eta_2=\eta$.
We obtain finitely many pieces
$\chi_{1,P_1}\chi_{2,P_2}g$, supported in
$P_1\times P_2$, with
\[
    \ell(P_i)=\eta\ell(Q_i),
    \qquad AP_i\subset Q_i^{**},
\]
and
\begin{equation}\label{eq:localized-pieces-sum}
    \sum_{P_1,P_2}
        \|\chi_{1,P_1}\chi_{2,P_2}g\|
            _{h^1_{\eta\ell(Q_1),\eta\ell(Q_2)}}
    \leq
    C_\eta\|g\|_{h^1_{\ell(Q_1),\ell(Q_2)}}.
\end{equation}
Apply Lemma~\ref{lem:localized-atomic-decomposition}
to every piece at scales
$\eta\ell(Q_1),\eta\ell(Q_2)$.
Its atoms have associated open sets and
rectangular components inside
\[
    AP_1\times AP_2
    \subset Q_1^{**}\times Q_2^{**}.
\]
Combining the finitely many decompositions
and using \eqref{eq:localized-pieces-sum} gives
\[
    g=\sum_\nu\lambda_\nu a_\nu,
    \qquad
    \sum_\nu|\lambda_\nu|
    \leq C_\eta\|g\|_{h^1_{\ell(Q_1),\ell(Q_2)}}.
\]
Initially, the series converges in
$h^1_{\eta\ell(Q_1),\eta\ell(Q_2)}$ and in $L^1$.

By \eqref{eq:local-scale-comparison-reverse}
and \eqref{eq:local-product-atom-bound},
\begin{align*}
    \|a_\nu\|_{h^1_{\ell(Q_1),\ell(Q_2)}}
    &\leq
    C\left(1+\log\frac1\eta\right)^2
    \|a_\nu\|_{h^1_{\eta\ell(Q_1),\eta\ell(Q_2)}}\leq
    C\left(1+\log\frac1\eta\right)^2.
\end{align*}
Since $\eta$ is fixed, this proves
\eqref{eq:localized-atom-original-scales}.
The series therefore converges absolutely
at the original Hardy scales as well.
Its limit is $g$ by
\eqref{eq:product-local-L1-embedding}
and uniqueness of the $L^1$ limit.

If $g\in L^2$, each localized piece belongs
to $L^2$.
Choose its decomposition to converge in $L^2$
and combine the finitely many series.
This gives the required $L^2$ convergence.
\end{proof}

\section{Comparison with Schr\"odinger Hardy spaces}
\label{sec:comparison}

Throughout this section, we assume
\eqref{ass:F}, \eqref{ass:K}, and \eqref{ass:D}.
We prove that the product square-function norm
is equivalent, on the defining $L^2$ class,
to the sum of the localized classical Hardy norms.
We first establish the required one-parameter
estimates and then apply them successively
in the two variables.

The coverings and partitions of unity are
those of Definition~\ref{def:coverings} and
Subsection~\ref{subsec:partitions}.
The square functions are defined in
\eqref{eq:one-parameter-square-function} and
\eqref{eq:product-square-function}, and the
local Hardy norms are those of
\eqref{eq:B} and
\eqref{eq:local-product-Hardy-norm}.

\subsection{One-parameter atomic characterizations}
\label{subsec:one-parameter-characterizations}

We first work with a single operator
$L=-\Delta+V$ on $\mathbb R^n$ and suppress
the index $i$.
Write
\[
    M_Lg(x)=\sup_{t>0}|e^{-tL}g(x)|.
\]
We use the auxiliary atomic space from
\cite[pp.~39--40, equations~(1.2)--(1.3)]{DZ_Studia}.

\begin{definition}\label{def:one-parameter-covering-atoms}
An $H^1_{\mathcal Q}$-atom is a measurable
function $a$ of one of the following types:
\begin{enumerate}
\item
For some $Q\in\mathcal Q$ and an axis-parallel
cube $K\subset Q^{**}$,
\[
    a=0\quad\text{almost everywhere outside }K,
    \qquad
    \|a\|_\infty\leq|K|^{-1},
    \qquad
    \int_{\mathbb R^n}a(x)\,dx=0.
\]
\item
For some $Q\in\mathcal Q$,
\[
    a=|Q|^{-1}\mathbf1_Q.
\]
\end{enumerate}
The first type consists of classical
$(1,\infty)$-atoms; the second type has no
cancellation requirement.

The space $H^1_{\mathcal Q}$ consists of all
$g\in L^1$ admitting a representation
\[
    g=\sum_{\nu=1}^{\infty}\lambda_\nu a_\nu
    \quad\text{in }L^1,
    \qquad
    \sum_{\nu=1}^{\infty}|\lambda_\nu|<\infty,
\]
by these atoms. Its norm is
\[
    \|g\|_{H^1_{\mathcal Q}}
    =
    \inf\left\{
        \sum_{\nu=1}^{\infty}|\lambda_\nu|:
        g=\sum_{\nu=1}^{\infty}\lambda_\nu a_\nu
        \text{ in }L^1
    \right\}.
\]
\end{definition}

Every atom has $L^1$ norm at most one.
Thus each series with summable coefficients
converges absolutely in $L^1$, and
$
    \|g\|_1\leq\|g\|_{H^1_{\mathcal Q}}
$.

\begin{theorem}[Dziuba\'nski--Zienkiewicz]
\label{thm:DZ-maximal-atomic}
Under ~\eqref{ass:F}, ~\eqref{ass:K}, and ~\eqref{ass:D},
a function $g\in L^1$ belongs to $H^1_{\mathcal Q}$
if and only if $M_Lg\in L^1$.
Moreover,
\begin{equation}\label{eq:scalar-maximal-atomic}
    \|g\|_{H^1_{\mathcal Q}}
    \asymp\|M_Lg\|_1.
\end{equation}
\end{theorem}

This is \cite[Theorem~2.2, equation~(2.3)]{DZ_Studia}.
We record the changes of normalization needed
to apply that result.
The enlargement there is $Q^*=(1+\beta)Q$;
taking $\beta=1$ gives
\eqref{eq:cube-enlargements}.
Their condition ~\eqref{ass:D} is imposed for
$y\in Q^*$, whereas ours holds on $Q^{****}$.

The authors use $\operatorname{diam}(Q)$
instead of $\ell(Q)$.
Since
$\operatorname{diam}(Q)^2=n\ell(Q)^2$,
our condition ~\eqref{ass:D} implies theirs by
monotonicity of the semigroup mass in time.
Indeed, symmetry and the semigroup property give
\[
    \int_{\mathbb R^n}k_t(x,y)\,dx
    =
    (e^{-tL}\mathbf1)(y),
\]
and this quantity is nonincreasing in $t$
because the semigroup is sub-Markovian.

For condition ~\eqref{ass:K}, put
$W=\mathbf1_{Q^{***}}V$.
Our assumption gives
\[
    \left\|
        \int_0^{2\ell(Q)^2}H_sW\,ds
    \right\|_\infty
    \leq C_K.
\]
For every integer $m\geq0$, the semigroup
property and $L^\infty$ contractivity imply
\[
    \left\|
        \int_{2m\ell(Q)^2}^{2(m+1)\ell(Q)^2}
            H_sW\,ds
    \right\|_\infty
    \leq C_K.
\]
Hence the integral up to $2t$ is bounded by
$nC_K$ for $\ell(Q)^2<t\leq n\ell(Q)^2$.
On this interval,
$t/\operatorname{diam}(Q)^2\geq1/n$.
For $t\leq\ell(Q)^2$, the change from side
length to diameter only multiplies the
constant by $n^\delta$.
Thus ~\eqref{ass:K} also holds in the
normalization of \cite{DZ_Studia}.

The Gaussian bound \eqref{eq:heat-domination}
allows us to use the maximal characterization
of Song and Yan.
More precisely, \cite[Theorem~1.3]{SongYan2018},
together with part~\textup{(ii)} of the concluding
Remarks in that paper, gives the following
square-function comparison.

\begin{lemma}\label{lem:scalar-square-maximal}
For every $g\in L^2(\mathbb R^n)$,
\[
    S_Lg\in L^1
    \quad\Longleftrightarrow\quad
    M_Lg\in L^1,
\]
and
\begin{equation}\label{eq:scalar-square-maximal}
    \|S_Lg\|_1\asymp\|M_Lg\|_1.
\end{equation}
\end{lemma}

The square function in \cite{SongYan2018}
uses $dy\,dt/(t|B(y,t)|)$, which differs from
our measure only by the constant $|B(0,1)|^{-1}$.
Their radial maximal function uses $e^{-t^2L}$;
its supremum over $t>0$ is exactly $M_L$.

\begin{corollary}\label{cor:scalar-square-atomic}
For every $g\in L^2$,
\[
    S_Lg\in L^1
    \quad\Longleftrightarrow\quad
    g\in H^1_{\mathcal Q}.
\]
In this case,
\begin{equation}\label{eq:scalar-square-atomic}
    \|g\|_{H^1_{\mathcal Q}}
    \asymp\|S_Lg\|_1.
\end{equation}
\end{corollary}

\begin{proof}
If $S_Lg\in L^1$, then
Lemma~\ref{lem:scalar-square-maximal} gives
$M_Lg\in L^1$.
Strong continuity of the semigroup on $L^2$
provides a sequence $t_j\downarrow0$ such that
$e^{-t_jL}g\to g$ almost everywhere.
Thus $|g|\leq M_Lg$ almost everywhere, so
$g\in L^1$ and
Theorem~\ref{thm:DZ-maximal-atomic} applies.
Combining \eqref{eq:scalar-maximal-atomic}
and \eqref{eq:scalar-square-maximal} proves
one implication and its norm estimate.
The reverse implication follows from the
same two estimates.
\end{proof}

\subsection{One-parameter localization estimates}
\label{subsec:one-parameter-comparison}

We next establish the geometric and local
Hardy space estimates used in the comparison.
The assumptions on the potential enter only
through Corollary~\ref{cor:scalar-square-atomic}.

\begin{lemma}\label{lem:neighbouring-cutoffs}
Assume ~\eqref{ass:F}.
For every $Q\in\mathcal Q$, there are uniformly
finitely many $P\in\mathcal Q$ satisfying
$P^*\cap Q^{**}\neq\varnothing$.
For each such pair,
\begin{equation}\label{eq:neighbouring-cube-scales}
    C_F^{-1}\ell(Q)\leq\ell(P)\leq C_F\ell(Q),
\end{equation}
and
\begin{equation}\label{eq:neighbouring-cutoff-estimate}
    \|\phi_Pu\|_{h^1_{\ell(P)}}
    \leq C\|u\|_{h^1_{\ell(Q)}},
    \qquad u\in h^1_{\ell(Q)}.
\end{equation}
The constants are independent of $P$, $Q$, and $u$.
\end{lemma}

\begin{proof}
The intersection condition implies
$P^{****}\cap Q^{****}\neq\varnothing$,
so \eqref{ass:F} gives
\eqref{eq:neighbouring-cube-scales}.
All such cubes $P$ lie in a fixed concentric
enlargement of $Q$, have pairwise disjoint
interiors, and satisfy
$|P|\geq C_F^{-n}|Q|$.
Comparison of volumes bounds their number.

By \eqref{eq:one-parameter-scale},
\[
    \|u\|_{h^1_{\ell(P)}}
    \leq C\|u\|_{h^1_{\ell(Q)}};
\]
when $\ell(P)>\ell(Q)$, the logarithmic factor
is bounded by $1+\log C_F$.
Moreover,
\eqref{eq:partition-of-unity} and
\eqref{eq:partition-derivative-bounds} imply
\[
    \|\phi_P\|_\infty
    +\ell(P)\|\nabla\phi_P\|_\infty\leq C.
\]
Applying \eqref{eq:one-parameter-cutoff}
at scale $\ell(P)$ proves
\eqref{eq:neighbouring-cutoff-estimate}.
\end{proof}

\begin{lemma}\label{lem:local-to-covering-atoms}
If $P\in\mathcal Q$ and $u\in h^1_{\ell(P)}$
vanishes almost everywhere outside $P^*$, then
\begin{equation}\label{eq:local-to-covering-atoms}
    \|u\|_{H^1_{\mathcal Q}}
    \leq C\|u\|_{h^1_{\ell(P)}}.
\end{equation}
Under ~\eqref{ass:F}, the more general estimate
\begin{equation}\label{eq:supported-local-to-covering}
    \|u\|_{H^1_{\mathcal Q}}
    \leq C\|u\|_{h^1_{\ell(Q)}}
\end{equation}
holds whenever $Q\in\mathcal Q$ and
$u\in h^1_{\ell(Q)}$ vanishes almost everywhere
outside $Q^{**}$.
The constants are independent of the cubes and $u$.
\end{lemma}

\begin{proof}
Apply \cite[Theorem~3.2]{DZ_Studia} to a uniform
partition into cubes of side length $\ell(P)$
containing $P$.
Together with \eqref{eq:local-heat-characterization},
this gives a decomposition of $u$ into classical
atoms supported in $P^{**}$ and normalized
indicators $|T|^{-1}\mathbf1_T$, where
$T\subset P^{**}$ and $|T|=|P|$.
The coefficient sum is bounded by
$C\|u\|_{h^1_{\ell(P)}}$.

The classical atoms are already
$H^1_{\mathcal Q}$-atoms.
For each remaining term, write
\[
    \frac{\mathbf1_T}{|T|}
    =
    \left(
        \frac{\mathbf1_T}{|T|}
        -
        \frac{\mathbf1_P}{|P|}
    \right)
    +
    \frac{\mathbf1_P}{|P|}.
\]
The difference has integral zero, is supported
in $P^{**}$, and has $L^\infty$ norm at most
$2|P|^{-1}=2\cdot4^n|P^{**}|^{-1}$.
It is therefore a fixed multiple of an
$H^1_{\mathcal Q}$-atom.
The last term is itself an atom.
This proves \eqref{eq:local-to-covering-atoms}.

If $u$ is supported in $Q^{**}$, then
\[
    u=
    \sum_{{P\in\mathcal Q\,
                    P^*\cap Q^{**}\neq\varnothing}}
        \phi_Pu.
\]
By Lemma~\ref{lem:neighbouring-cutoffs}, this
sum has uniformly finitely many terms.
Applying \eqref{eq:local-to-covering-atoms}
and \eqref{eq:neighbouring-cutoff-estimate}
to each term proves
\eqref{eq:supported-local-to-covering}.
\end{proof}

\begin{lemma}\label{lem:covering-atomic-localization}
Assume ~\eqref{ass:F}.
Every $H^1_{\mathcal Q}$-atom associated with
$Q\in\mathcal Q$ satisfies
\begin{equation}\label{eq:covering-atom-local-bound}
    \|a\|_{h^1_{\ell(Q)}}\leq C.
\end{equation}
Moreover,
\begin{equation}\label{eq:covering-atomic-localization}
    \sum_{P\in\mathcal Q}
        \|\phi_Pg\|_{h^1_{\ell(P)}}
    \leq C\|g\|_{H^1_{\mathcal Q}},
    \qquad g\in H^1_{\mathcal Q}.
\end{equation}
\end{lemma}

\begin{proof}
For a classical atom, the global heat maximal
function has uniformly bounded $L^1$ norm.
Since
\[
    \sup_{0<t\leq\ell(Q)^2}|H_ta|
    \leq\sup_{t>0}|H_ta|,
\]
\eqref{eq:local-heat-characterization} gives
\eqref{eq:covering-atom-local-bound}.

For $a=|Q|^{-1}\mathbf1_Q$, translation and
dilation reduce the assertion to a unit cube.
On a fixed enlargement of that cube,
Cauchy--Schwarz and the $L^2$ boundedness of
the heat maximal operator give the $L^1$ bound.
Outside the enlargement, the Gaussian estimate
for $0<t\leq1$ gives an integrable bound.
This proves \eqref{eq:covering-atom-local-bound}
for both types of atoms.

Since $a$ is supported in $Q^{**}$, only cubes
$P$ with $P^*\cap Q^{**}\neq\varnothing$ can
contribute to $\phi_Pa$.
By Lemma~\ref{lem:neighbouring-cutoffs} and
\eqref{eq:covering-atom-local-bound},
\[
    \sum_{P\in\mathcal Q}
        \|\phi_Pa\|_{h^1_{\ell(P)}}
    \leq C.
\]

Now choose an atomic representation
$g=\sum_\nu\lambda_\nu a_\nu$ in $L^1$,
and let $g_N=\sum_{\nu=1}^N\lambda_\nu a_\nu$.
The triangle inequality gives
\[
    \sum_{P\in\mathcal Q}
        \|\phi_Pg_N\|_{h^1_{\ell(P)}}
    \leq C\sum_{\nu=1}^N|\lambda_\nu|.
\]
For each $P$, boundedness of $\phi_P$ implies
$\phi_Pg_N\to\phi_Pg$ in $L^1$.
The one-parameter assertion of
Lemma~\ref{lem:local-Hardy-Fatou}, followed
by Fatou's lemma for the sum over $P$, yields
\[
    \sum_{P\in\mathcal Q}
        \|\phi_Pg\|_{h^1_{\ell(P)}}
    \leq
    \liminf_{N\to\infty}
        \sum_{P\in\mathcal Q}
            \|\phi_Pg_N\|_{h^1_{\ell(P)}}
    \leq C\sum_\nu|\lambda_\nu|.
\]
Taking the infimum over representations proves
\eqref{eq:covering-atomic-localization}.
\end{proof}

\begin{proposition}\label{prop:one-parameter-comparison}
Assume \eqref{ass:F}, \eqref{ass:K}, and
\eqref{ass:D}.

For every $g\in L^2$ with $S_Lg\in L^1$,
there is an atomic decomposition
\[
    g=\sum_{\nu=1}^{\infty}\lambda_\nu a_\nu
    \quad\text{in }L^1,
    \qquad
    \sum_{\nu=1}^{\infty}|\lambda_\nu|
    \leq C\|S_Lg\|_1,
\]
where the $a_\nu$ are
$H^1_{\mathcal Q}$-atoms from
Definition~\ref{def:one-parameter-covering-atoms}.
Moreover,
\begin{equation}\label{eq:one-parameter-localization-proved}
    \sum_{Q\in\mathcal Q}
        \|\phi_Qg\|_{h^1_{\ell(Q)}}
    \leq C\|S_Lg\|_1.
\end{equation}

If $Q\in\mathcal Q$ and
$g\in L^2\cap h^1_{\ell(Q)}$ vanishes almost
everywhere outside $Q^{**}$, then $g$ admits
an $L^1$-convergent decomposition into
$H^1_{\mathcal Q}$-atoms with
\[
    \sum_{\nu=1}^{\infty}|\lambda_\nu|
    \leq C\|g\|_{h^1_{\ell(Q)}}.
\]
In particular,
\begin{equation}\label{eq:one-parameter-supported-proved}
    \|S_Lg\|_1
    \leq C\|g\|_{h^1_{\ell(Q)}}.
\end{equation}
All constants are independent of $Q$ and $g$.
\end{proposition}

\begin{proof}
Under \eqref{ass:F}, \eqref{ass:K},
and \eqref{ass:D}, the theorem of
Dziuba\'nski and Zienkiewicz
\cite[Theorem~2.2]{DZ_Studia},
recalled in Theorem~\ref{thm:DZ-maximal-atomic},
identifies $H^1_{\mathcal Q}$ with the Hardy
space defined by $M_L$.
Together with \eqref{eq:scalar-square-maximal},
this gives \eqref{eq:scalar-square-atomic}.
The assumptions on the potential,
\eqref{ass:K} and \eqref{ass:D}, enter through
this characterization; \eqref{ass:F} is also
used in the localization estimates below.

First, let $g\in L^2$ with $S_Lg\in L^1$.
By \eqref{eq:scalar-square-atomic},
\[
    g\in H^1_{\mathcal Q},
    \qquad
    \|g\|_{H^1_{\mathcal Q}}
    \leq C\|S_Lg\|_1.
\]
If $g\neq0$, the definition of the atomic
norm allows us to choose a representation
\[
    g=\sum_{\nu=1}^{\infty}\lambda_\nu a_\nu
    \quad\text{in }L^1,
    \qquad
    \sum_{\nu=1}^{\infty}|\lambda_\nu|
    \leq2\|g\|_{H^1_{\mathcal Q}}
    \leq C\|S_Lg\|_1.
\]
For $g=0$, take all coefficients equal to zero.
Furthermore,
\eqref{eq:covering-atomic-localization} gives
\[
    \sum_{Q\in\mathcal Q}
        \|\phi_Qg\|_{h^1_{\ell(Q)}}
    \leq C\|g\|_{H^1_{\mathcal Q}}
    \leq C\|S_Lg\|_1,
\]
which proves
\eqref{eq:one-parameter-localization-proved}. Now let $Q\in\mathcal Q$ and
$g\in L^2\cap h^1_{\ell(Q)}$ vanish almost
everywhere outside $Q^{**}$.
By \eqref{eq:supported-local-to-covering},
\[
    g\in H^1_{\mathcal Q},
    \qquad
    \|g\|_{H^1_{\mathcal Q}}
    \leq C\|g\|_{h^1_{\ell(Q)}}.
\]
Choosing an atomic representation as above
therefore gives
\[
    g=\sum_{\nu=1}^{\infty}\lambda_\nu a_\nu
    \quad\text{in }L^1,
    \qquad
    \sum_{\nu=1}^{\infty}|\lambda_\nu|
    \leq C\|g\|_{h^1_{\ell(Q)}}.
\]
Finally, \eqref{eq:scalar-square-atomic} yields
\[
    \|S_Lg\|_1
    \leq C\|g\|_{H^1_{\mathcal Q}}
    \leq C\|g\|_{h^1_{\ell(Q)}},
\]
proving \eqref{eq:one-parameter-supported-proved}.
\end{proof}
\subsection{Passage to the product setting}
\label{subsec:product-comparison}

We now apply Proposition~\ref{prop:one-parameter-comparison}
to $L_1$ and $L_2$.
We first record its sequence-valued form. For $G=(g_\nu)_{\nu\geq1}$, the operators act
coordinatewise, and
\[
    S_{L_i}G(x_i)
    =
    \left(
        \int_0^\infty
        \int_{|y_i-x_i|<t}
            \sum_{\nu\geq1}
                |t^2L_i e^{-t^2L_i}g_\nu(y_i)|^2
            \frac{dy_i\,dt}{t^{n_i+1}}
    \right)^{1/2}.
\]
Tonelli's theorem gives
\begin{equation}\label{eq:sequence-square-function}
    S_{L_i}G
    =
    \left(
        \sum_{\nu\geq1}|S_{L_i}g_\nu|^2
    \right)^{1/2}.
\end{equation}
The local operators $\mathcal B_{i,r}$ are
extended coordinatewise as well.

\begin{lemma}\label{lem:sequence-comparison}
For $i=1,2$ and every sequence
$(g_\nu)_{\nu\geq1}$ in $L^2(\mathbb R^{n_i})$,
\begin{equation}\label{eq:vector-localization-needed}
\begin{split}
    &\sum_{Q\in\mathcal Q_i}
    \int_{\mathbb R^{n_i}}
        \left(
            \sum_{\nu\geq1}
                \|\mathcal B_{i,\ell(Q)}
                    (\phi_Qg_\nu)(x_i)\|_{\ell^2}^2
        \right)^{1/2}dx_i\leq
    C\int_{\mathbb R^{n_i}}
        \left(
            \sum_{\nu\geq1}|S_{L_i}g_\nu(x_i)|^2
        \right)^{1/2}dx_i.
\end{split}
\end{equation}
If every $g_\nu$ vanishes almost everywhere
outside the same cube $Q^{**}$, then
\begin{equation}\label{eq:vector-supported-needed}
\begin{split}
    &\int_{\mathbb R^{n_i}}
        \left(
            \sum_{\nu\geq1}|S_{L_i}g_\nu(x_i)|^2
        \right)^{1/2}dx_i\leq
    C\int_{\mathbb R^{n_i}}
        \left(
            \sum_{\nu\geq1}
                \|\mathcal B_{i,\ell(Q)}
                    g_\nu(x_i)\|_{\ell^2}^2
        \right)^{1/2}dx_i.
\end{split}
\end{equation}
\end{lemma}

\begin{proof}
We first consider a finite sequence
$g_1,\ldots,g_N$ with finite right-hand side.
For $\tau\in[0,1]$, apply
\eqref{eq:one-parameter-localization-proved}
to
$
    \sum_{\nu=1}^N r_\nu(\tau)g_\nu
$
and integrate in $\tau$.
This application is legitimate because
the finite right-hand side implies
$S_{L_i}g_\nu\in L^1$ for every $\nu$,
and hence the same holds for each finite
Rademacher sum.

On the left, for each fixed $Q$ and $x_i$,
apply \eqref{eq:sequence-khintchine} to
the vectors
$
    \mathcal B_{i,\ell(Q)}
        (\phi_Qg_\nu)(x_i)\in\ell^2.
$.
On the right, for each fixed $x_i$,
apply the same inequality to the functions
$
    (y_i,t)\longmapsto
    t^2L_i e^{-t^2L_i}g_\nu(y_i)
$
in the Hilbert space
$
    L^2\big(
        \Gamma_i(x_i),
        \frac{dy_i\,dt}{t^{n_i+1}}
    \big)
$.
Their norms are precisely $S_{L_i}g_\nu(x_i)$.
Thus the argument of
Lemma~\ref{lem:sequence-extension},
together with Tonelli's theorem for the
spatial integrals and the sum over $Q$,
gives \eqref{eq:vector-localization-needed}.

For \eqref{eq:vector-supported-needed},
apply the same argument to
\eqref{eq:one-parameter-supported-proved}.
Every Rademacher sum vanishes outside
$Q^{**}$.
Moreover, each $g_\nu$ belongs to $L^1$,
since it belongs to $L^2$ and is supported
in a set of finite measure.
Finiteness of the right-hand side gives
\[
    \|\mathcal B_{i,\ell(Q)}g_\nu\|_{L^1(\ell^2)}
    <\infty,
\]
so each $g_\nu$, and therefore every finite
Rademacher sum, belongs to $h^1_{\ell(Q)}$.
The scalar estimate is thus applicable.
Integrating in $\tau$ and using
\eqref{eq:sequence-khintchine} for the cone
norm on the left and the $\ell^2$ norm on
the right proves the desired inequality.

The constants are independent of $N$.
Applying the finite-sequence estimates to
the first $N$ terms and letting $N\to\infty$,
we obtain both assertions for infinite
sequences by monotone convergence.
\end{proof}

\begin{lemma}\label{lem:product-localization}
For every $f\in L^2$ with $S_{L_1,L_2}f\in L^1$,
\begin{equation}\label{eq:product-localization-upper}
    \sum_{Q_1\in\mathcal Q_1}
    \sum_{Q_2\in\mathcal Q_2}
        \|\phi_{Q_1}\phi_{Q_2}f\|
            _{h^1_{\ell(Q_1),\ell(Q_2)}}
    \leq C\|S_{L_1,L_2}f\|_1.
\end{equation}
\end{lemma}

\begin{proof}
We first localize in $x_2$, retaining the
cone variables associated with $x_1$.
An orthonormal expansion allows us to use
\eqref{eq:vector-localization-needed}.

\medskip
\noindent
\textbf{Step 1: Coefficients in the first cone.} By \eqref{eq:schrodinger-square-L2} and
Fubini's theorem, for almost every fixed $x_1$,
\[
    \int_{\mathbb R^{n_2}}
    \iint_{\Gamma_1(x_1)}
        |(t_1^2L_1e^{-t_1^2L_1}f)(y_1,x_2)|^2
        \frac{dy_1\,dt_1}{t_1^{n_1+1}}\,dx_2
    <\infty.
\]
Fix such an $x_1$ and choose an orthonormal
basis $(e_\nu)_{\nu\geq1}$ of
$
    L^2\big(
        \Gamma_1(x_1),
        \frac{dy_1\,dt_1}{t_1^{n_1+1}}
    \big)
$. Define
\[
    g_\nu(x_2)
    =
    \iint_{\Gamma_1(x_1)}
        (t_1^2L_1e^{-t_1^2L_1}f)(y_1,x_2)
        \overline{e_\nu(y_1,t_1)}
        \frac{dy_1\,dt_1}{t_1^{n_1+1}}.
\]
Each $g_\nu$ belongs to $L^2(\mathbb R^{n_2})$,
because taking a coefficient against a unit
vector is bounded on the cone $L^2$ space.
The dependence on the fixed $x_1$ is suppressed.

\medskip
\noindent
\textbf{Step 2: Identification by Parseval's identity.} The operator $t_2^2L_2e^{-t_2^2L_2}$ acts only
in the second variable and therefore commutes
with taking coefficients against $e_\nu$.
Thus
\begin{align*}
    &(t_2^2L_2e^{-t_2^2L_2}g_\nu)(y_2)=
    \iint_{\Gamma_1(x_1)}
        \big[
            (t_1^2L_1e^{-t_1^2L_1})
            (t_2^2L_2e^{-t_2^2L_2})f
        \big](y_1,y_2)
        \overline{e_\nu(y_1,t_1)}
        \frac{dy_1\,dt_1}{t_1^{n_1+1}}.
\end{align*}
This equality holds in $L^2$ in the second
variable, since the operations involved are
bounded on the relevant $L^2$ spaces.

Parseval's identity gives
\begin{align*}
    &\sum_{\nu\geq1}
        |(t_2^2L_2e^{-t_2^2L_2}g_\nu)(y_2)|^2=
    \iint_{\Gamma_1(x_1)}
        \left|
            (t_1^2L_1e^{-t_1^2L_1})
            (t_2^2L_2e^{-t_2^2L_2})f(y_1,y_2)
        \right|^2
        \frac{dy_1\,dt_1}{t_1^{n_1+1}}.
\end{align*}
Integrating over $\Gamma_2(x_2)$ and using
Tonelli's theorem yields
\begin{equation}\label{eq:localization-parseval-square}
    \left(
        \sum_{\nu\geq1}|S_{L_2}g_\nu(x_2)|^2
    \right)^{1/2}
    =
    S_{L_1,L_2}f(x_1,x_2).
\end{equation}

Fix $Q_2\in\mathcal Q_2$.
Multiplication by $\phi_{Q_2}$ and convolution
with $\psi^{(2)}_{k,\ell(Q_2)}$ also act only
in the second variable. Hence
\begin{align*}
    &\bigl(
        \psi^{(2)}_{k,\ell(Q_2)}
        *(\phi_{Q_2}g_\nu)
    \bigr)(x_2)=
    \iint_{\Gamma_1(x_1)}
        \Big[
            t_1^2L_1e^{-t_1^2L_1}
            \bigl(
                \psi^{(2)}_{k,\ell(Q_2)}
                *_2(\phi_{Q_2}f)
            \bigr)
        \Big](y_1,x_2)
        \overline{e_\nu(y_1,t_1)}
        \frac{dy_1\,dt_1}{t_1^{n_1+1}}.
\end{align*}
For each fixed $k$, Parseval's identity replaces
the sum of squared coefficients over $\nu$
by the cone integral of the squared integrand.
Summing next over $k$ and applying Tonelli's
theorem gives
\begin{equation}\label{eq:localization-parseval-local}
\begin{split}
    &\left(
        \sum_{\nu\geq1}
            \|\mathcal B_{2,\ell(Q_2)}
                (\phi_{Q_2}g_\nu)(x_2)\|_{\ell^2}^2
    \right)^{1/2}=
    S_{L_1}\bigl(
        \mathcal B_{2,\ell(Q_2)}(\phi_{Q_2}f)
    \bigr)(x_1,x_2).
\end{split}
\end{equation}
On the right, $S_{L_1}$ acts on the sequence
indexed by the local kernel index $k$.

\medskip
\noindent
\textbf{Step 3: Localization in the second variable.} By \eqref{eq:localization-parseval-square}
and $S_{L_1,L_2}f\in L^1$, the sequence
$(g_\nu)$ has finite right-hand side in
\eqref{eq:vector-localization-needed}
for almost every $x_1$.
Apply that estimate in $x_2$, substitute
\eqref{eq:localization-parseval-local} and
\eqref{eq:localization-parseval-square}, and
integrate in $x_1$.
We obtain
\begin{equation}\label{eq:localize-second-variable}
    \sum_{Q_2\in\mathcal Q_2}
        \left\|
            S_{L_1}\bigl(
                \mathcal B_{2,\ell(Q_2)}(\phi_{Q_2}f)
            \bigr)
        \right\|_1
    \leq C\|S_{L_1,L_2}f\|_1.
\end{equation}
Although the basis was chosen for fixed $x_1$,
the expressions after the Parseval identities
are the measurable square functions displayed
above; no choice of bases measurable in $x_1$
is needed.

\medskip
\noindent
\textbf{Step 4: Localization in the first variable.} Fix $Q_2$ and, for almost every $x_2$, set
\[
    g_k(x_1)
    =
    \big(
        \psi^{(2)}_{k,\ell(Q_2)}
        *_2(\phi_{Q_2}f)
    \big)(x_1,x_2).
\]
These functions belong to $L^2(\mathbb R^{n_1})$
by Young's inequality and Fubini's theorem.
Equation~\eqref{eq:localize-second-variable}
and \eqref{eq:sequence-square-function} show that
\[
    \int_{\mathbb R^{n_1}}
        \left(
            \sum_{k\geq0}|S_{L_1}g_k(x_1)|^2
        \right)^{1/2}dx_1<\infty
\]
for almost every $x_2$.
We may therefore apply
\eqref{eq:vector-localization-needed} in $x_1$.

Since $\phi_{Q_1}$ depends only on $x_1$,
\[
    \phi_{Q_1}g_k
    =
    \bigl(
        \psi^{(2)}_{k,\ell(Q_2)}
        *_2(\phi_{Q_1}\phi_{Q_2}f)
    \bigr)(\cdot,x_2).
\]

Taking the squared $\ell^2$ norm sums over $j$;
summing these norms over $k$ therefore gives
\begin{align*}
    &\sum_{k\geq0}
        \|\mathcal B_{1,\ell(Q_1)}
            (\phi_{Q_1}g_k)(x_1)\|_{\ell^2}^2=
    \sum_{j,k\geq0}
        \left|
            \bigl(
                \psi^{(1)}_{j,\ell(Q_1)}*_1
                \psi^{(2)}_{k,\ell(Q_2)}*_2
                (\phi_{Q_1}\phi_{Q_2}f)
            \bigr)(x_1,x_2)
        \right|^2.
\end{align*}
By \eqref{eq:local-product-square}, the square
root of this expression is
$S^{\mathrm{loc}}_{\ell(Q_1),\ell(Q_2)}
(\phi_{Q_1}\phi_{Q_2}f)(x_1,x_2)$.
Likewise, by the definition of the square
function for sequence-valued inputs,
\[
    \left(
        \sum_{k\geq0}|S_{L_1}g_k(x_1)|^2
    \right)^{1/2}
    =
    S_{L_1}\bigl(
        \mathcal B_{2,\ell(Q_2)}(\phi_{Q_2}f)
    \bigr)(x_1,x_2).
\]

Applying \eqref{eq:vector-localization-needed}
in the first variable to $(g_k)_{k\geq0}$,
and substituting these two identities, we obtain
\[
    \sum_{Q_1\in\mathcal Q_1}
    \int_{\mathbb R^{n_1}}
        S^{\mathrm{loc}}_{\ell(Q_1),\ell(Q_2)}
        (\phi_{Q_1}\phi_{Q_2}f)(x_1,x_2)\,dx_1
    \leq
    C\int_{\mathbb R^{n_1}}
        S_{L_1}\bigl(
            \mathcal B_{2,\ell(Q_2)}(\phi_{Q_2}f)
        \bigr)(x_1,x_2)\,dx_1.
\]
This inequality holds for almost every $x_2$.
We integrate both sides over
$x_2\in\mathbb R^{n_2}$ and use Tonelli's
theorem to interchange the sum over $Q_1$
with the spatial integrals.
By \eqref{eq:local-product-Hardy-norm}, the
integral of each local product square function
is its local Hardy norm. Hence
\[
    \sum_{Q_1\in\mathcal Q_1}
        \|\phi_{Q_1}\phi_{Q_2}f\|
            _{h^1_{\ell(Q_1),\ell(Q_2)}}
    \leq
    C\left\|
        S_{L_1}\bigl(
            \mathcal B_{2,\ell(Q_2)}(\phi_{Q_2}f)
        \bigr)
    \right\|_1,
\]
where the $L^1$ norm on the right is taken
over $\mathbb R^{n_1}\times\mathbb R^{n_2}$.
Each localized function is in $L^1$, since it
is in $L^2$ and has bounded support.
Thus the local Hardy norms here are understood
in the sense of
\eqref{eq:local-product-Hardy-norm}.
Summing over $Q_2$ and applying
\eqref{eq:localize-second-variable} proves
\eqref{eq:product-localization-upper}.
\end{proof}

\begin{lemma}\label{lem:product-supported}
Let $Q_i\in\mathcal Q_i$, $i=1,2$.
For every $g\in L^2$ vanishing almost everywhere
outside $Q_1^{**}\times Q_2^{**}$,
\begin{equation}\label{eq:product-supported-comparison}
    \|S_{L_1,L_2}g\|_1
    \leq C\|g\|_{h^1_{\ell(Q_1),\ell(Q_2)}}.
\end{equation}
The constant is independent of the cubes and $g$.
\end{lemma}

\begin{proof}
We follow the argument of
Lemma~\ref{lem:product-localization}, using
\eqref{eq:vector-supported-needed} in place of
\eqref{eq:vector-localization-needed}.
This gives the comparison in the opposite
direction: we control the Schr\"odinger square
function by the local product Hardy norm.
At each application, we verify that the
functions satisfy the required support condition.

We may assume that the right-hand side is finite.

\medskip
\noindent
\textbf{Step 1: The supported estimate in the second variable.} For almost every $x_1$, apply
\eqref{eq:vector-supported-needed} in $x_2$
to the components of
$\mathcal B_{1,\ell(Q_1)}g$.
They vanish outside $Q_2^{**}$, since convolution
in $x_1$ does not change support in $x_2$.
Their $L^2$ integrability follows from Young's
inequality and Fubini's theorem.
The right-hand side of the sequence estimate
is integrable for almost every $x_1$, since
its integral in $x_1$ is
$\|g\|_{h^1_{\ell(Q_1),\ell(Q_2)}}$.
Consequently,
\begin{equation}\label{eq:supported-second-variable}
    \left\|
        S_{L_2}\bigl(\mathcal B_{1,\ell(Q_1)}g\bigr)
    \right\|_1
    \leq C\|g\|_{h^1_{\ell(Q_1),\ell(Q_2)}}.
\end{equation}

\medskip
\noindent
\textbf{Step 2: The supported estimate in the first variable.}

Use the coefficient construction from
Lemma~\ref{lem:product-localization}, with
the variables interchanged.
For almost every fixed $x_2$, choose an
orthonormal basis $(e_\nu)$ of
$
    L^2\big(
        \Gamma_2(x_2),
        \frac{dy_2\,dt_2}{t_2^{n_2+1}}
    \big)
$,
and set
\[
    g_\nu(x_1)
    =
    \iint_{\Gamma_2(x_2)}
        (t_2^2L_2e^{-t_2^2L_2}g)(x_1,y_2)
        \overline{e_\nu(y_2,t_2)}
        \frac{dy_2\,dt_2}{t_2^{n_2+1}}.
\]
By \eqref{eq:schrodinger-square-L2} and Fubini's
theorem, each $g_\nu$ belongs to $L^2$.
It vanishes outside $Q_1^{**}$, because the
operator inside the integral acts only in $x_2$.

The Parseval argument in
\eqref{eq:localization-parseval-square} and
\eqref{eq:localization-parseval-local}, with
the variables interchanged, gives
\begin{align*}
    \left(
        \sum_{\nu\geq1}|S_{L_1}g_\nu(x_1)|^2
    \right)^{1/2}
    &=
    S_{L_1,L_2}g(x_1,x_2),\\
    \left(
        \sum_{\nu\geq1}
            \|\mathcal B_{1,\ell(Q_1)}
                g_\nu(x_1)\|_{\ell^2}^2
    \right)^{1/2}
    &=
    S_{L_2}\bigl(
        \mathcal B_{1,\ell(Q_1)}g
    \bigr)(x_1,x_2).
\end{align*}
By \eqref{eq:supported-second-variable}, the
second expression is integrable in $x_1$
for almost every $x_2$.
Thus \eqref{eq:vector-supported-needed} applies.
Substituting these identities and integrating
in $x_2$ yields
\[
    \|S_{L_1,L_2}g\|_1
    \leq
    C\left\|
        S_{L_2}\bigl(\mathcal B_{1,\ell(Q_1)}g\bigr)
    \right\|_1.
\]
Combining this with
\eqref{eq:supported-second-variable}
proves \eqref{eq:product-supported-comparison}.
\end{proof}

\begin{theorem}\label{thm:product-localization-comparison}
Under \eqref{ass:F}, \eqref{ass:K}, and
\eqref{ass:D}, every
$f\in L^2(\mathbb R^{n_1}\times\mathbb R^{n_2})$
satisfies
\begin{equation}\label{eq:product-localization-comparison}
\begin{split}
    C^{-1}\|S_{L_1,L_2}f\|_1
    &\leq
    \sum_{Q_1\in\mathcal Q_1}
    \sum_{Q_2\in\mathcal Q_2}
        \|\phi_{Q_1}\phi_{Q_2}f\|
            _{h^1_{\ell(Q_1),\ell(Q_2)}}\leq C\|S_{L_1,L_2}f\|_1.
\end{split}
\end{equation}
The inequalities are understood with values
in $[0,\infty]$.
\end{theorem}

\begin{proof}
The second inequality is
Lemma~\ref{lem:product-localization} when
$\|S_{L_1,L_2}f\|_1$ is finite, and is immediate
otherwise.

For the first inequality, assume that the
sum of local norms is finite.
Choose increasing finite subfamilies
$\mathcal Q_i^{(N)}$ exhausting $\mathcal Q_i$,
and put
\[
    f_N
    =
    \sum_{{Q_1\in\mathcal Q_1^{(N)}\,
                    Q_2\in\mathcal Q_2^{(N)}}}
        \phi_{Q_1}\phi_{Q_2}f.
\]
By \eqref{eq:partition-of-unity} and
\eqref{eq:product-partition-of-unity},
$f_N\to f$ almost everywhere and
$|f_N-f|\leq|f|$.
Dominated convergence gives $f_N\to f$ in $L^2$.

Each $\phi_{Q_1}\phi_{Q_2}f$ is supported in
$Q_1^*\times Q_2^*\subset Q_1^{**}\times Q_2^{**}$.
The triangle inequality for the square function
and Lemma~\ref{lem:product-supported} imply
\[
    \|S_{L_1,L_2}f_N\|_1
    \leq
    C\sum_{{Q_1\in\mathcal Q_1^{(N)}\,
                    Q_2\in\mathcal Q_2^{(N)}}}
        \|\phi_{Q_1}\phi_{Q_2}f\|
            _{h^1_{\ell(Q_1),\ell(Q_2)}}.
\]
Moreover, the reverse triangle inequality and
\eqref{eq:schrodinger-square-L2} give
\[
    \|S_{L_1,L_2}f_N-S_{L_1,L_2}f\|_2
    \leq
    \|S_{L_1,L_2}(f_N-f)\|_2
    \leq C\|f_N-f\|_2
    \longrightarrow0.
\]
Along a subsequence, the square functions
therefore converge almost everywhere.
Fatou's lemma yields the first inequality
in \eqref{eq:product-localization-comparison}.
\end{proof}

\begin{corollary}\label{cor:schrodinger-core-L1}
If $f\in L^2$ and $S_{L_1,L_2}f\in L^1$, then
$f\in L^1$ and
\begin{equation}\label{eq:schrodinger-product-L1-embedding}
    \|f\|_1\leq C\|S_{L_1,L_2}f\|_1.
\end{equation}
\end{corollary}

\begin{proof}
Each localized function
$\phi_{Q_1}\phi_{Q_2}f$ belongs to $L^1$
by its bounded support and $L^2$ integrability.
By \eqref{eq:product-partition-of-unity},
Tonelli's theorem,
\eqref{eq:product-local-L1-embedding},
and \eqref{eq:product-localization-comparison},
\[
    \|f\|_1
    =
    \sum_{Q_1,Q_2}\|\phi_{Q_1}\phi_{Q_2}f\|_1
    \leq
    C\sum_{Q_1,Q_2}
        \|\phi_{Q_1}\phi_{Q_2}f\|
            _{h^1_{\ell(Q_1),\ell(Q_2)}}
    \leq C\|S_{L_1,L_2}f\|_1.
\]
The equality is initially understood in
$[0,\infty]$; the final bound proves integrability.
\end{proof}

\section{Proofs of the main results}
\label{sec:atomic-decomposition}

We prove Theorems~\ref{thm:main-atomic-characterization}
and~\ref{thm:main-localization}.
We first identify the Hardy space completion
with a subspace of $L^1$.
We then combine the comparison estimates from
Section~\ref{sec:comparison} with
Corollary~\ref{cor:atoms-inside-enlargements}
to obtain the atomic decomposition and the
localized norm characterization.

Throughout this section, $\eta$ is the fixed
constant chosen in
Corollary~\ref{cor:atoms-inside-enlargements}.
The atoms and the atomic space are those of
Definitions~\ref{def:covering-product-atom}
and~\ref{def:covering-product-atomic-space}.

\subsection{Identification of the completion}
\label{subsec:completion-identification}

Corollary~\ref{cor:schrodinger-core-L1}
provides an $L^1$ estimate on the defining
$L^2$ class.
To identify the completion with a space of
functions, we also need injectivity of the
resulting map into $L^1$.

\begin{lemma}\label{lem:completion-L1-identification}
The inclusion of the defining $L^2$ class
into $L^1$ extends to a continuous injective
map from $H^1_{L_1,L_2}$ into $L^1$.
Under this identification,
\begin{equation}\label{eq:schrodinger-completion-L1-embedding}
    \|f\|_1
    \leq C\|f\|_{H^1_{L_1,L_2}},
    \qquad f\in H^1_{L_1,L_2}.
\end{equation}
\end{lemma}

\begin{proof}
By \eqref{eq:schrodinger-product-L1-embedding},
every Cauchy sequence in the defining Hardy
norm is Cauchy in $L^1$.
Its $L^1$ limit depends only on its equivalence
class in the completion.
This defines a continuous linear map into
$L^1$ and gives
\eqref{eq:schrodinger-completion-L1-embedding}.
It remains to prove that the map is injective.

\medskip
\noindent
\textbf{Step 1: An estimate on bounded ranges of scales.} For $i=1,2$ and $t>0$, we have
\begin{equation}\label{eq:analysis-L1-L2-bound}
    \|t^2L_i e^{-t^2L_i}u\|_2
    \leq Ct^{-n_i/2}\|u\|_1.
\end{equation}
Indeed, write
\[
    t^2L_i e^{-t^2L_i}
    =
    \bigl(t^2L_i e^{-t^2L_i/2}\bigr)
    e^{-t^2L_i/2}.
\]
The first factor is uniformly bounded on
$L^2$ by the spectral theorem, while
\eqref{eq:heat-domination} gives
\[
    \|e^{-t^2L_i/2}u\|_2
    \leq Ct^{-n_i/2}\|u\|_1.
\]
The product heat kernel and the same
factorization in both variables therefore yield
\[
    \left\|
        (t_1^2L_1e^{-t_1^2L_1})
        (t_2^2L_2e^{-t_2^2L_2})u
    \right\|_2
    \leq
    Ct_1^{-n_1/2}t_2^{-n_2/2}\|u\|_1.
\]

Consequently, for every $0<\varepsilon<T<\infty$,
\begin{equation}\label{eq:truncated-analysis-continuity}
\begin{split}
    &\int_{\mathbb R^{n_1}\times\mathbb R^{n_2}}
    \int_\varepsilon^T\int_\varepsilon^T
    \int_{B_{\mathbb R^{n_1}}(0,1)}
    \int_{B_{\mathbb R^{n_2}}(0,1)}
    \Big|
        \big[
            (t_1^2L_1e^{-t_1^2L_1})
            (t_2^2L_2e^{-t_2^2L_2})u
        \big](x_1+t_1z_1,x_2+t_2z_2)
    \Big|^2\\
    &\hspace{25mm}\times
        dz_2\,dz_1\,
        \frac{dt_2}{t_2}\frac{dt_1}{t_1}\,
        dx_1\,dx_2
    \leq C_{\varepsilon,T}\|u\|_1^2.
\end{split}
\end{equation}
Here we first integrate in the spatial variables,
using translation invariance of Lebesgue measure,
and then integrate in the remaining variables.

\medskip
\noindent
\textbf{Step 2: Injectivity.} Let $(f_m)$ be a Cauchy sequence in the defining
Hardy norm such that $f_m\to0$ in $L^1$.
We must show that
\[
    \|S_{L_1,L_2}f_m\|_1\longrightarrow0.
\]

Set
\[
    F_m(x_1,x_2;t_1,z_1,t_2,z_2)
    =
    \big[
        (t_1^2L_1e^{-t_1^2L_1})
        (t_2^2L_2e^{-t_2^2L_2})f_m
    \big](x_1+t_1z_1,x_2+t_2z_2).
\]
For each spatial point, regard $F_m$ as an
element of
\[
    \mathcal H
    =
    L^2\left(
        (0,\infty)^2
        \times B_{\mathbb R^{n_1}}(0,1)
        \times B_{\mathbb R^{n_2}}(0,1),
        \frac{dt_1}{t_1}\frac{dt_2}{t_2}\,dz_1\,dz_2
    \right).
\]
By \eqref{eq:product-square-function} and
the changes of variables $y_i=x_i+t_iz_i$,
\[
    \|F_m-F_\ell\|_{L^1(\mathcal H)}
    =
    \|S_{L_1,L_2}(f_m-f_\ell)\|_1.
\]
Thus $F_m$ converges to some $F$ in
$L^1(\mathcal H)$.

Fix a bounded measurable spatial set $E$
and $0<\varepsilon<T<\infty$.
On the region
\[
    (x_1,x_2)\in E,\qquad
    \varepsilon\leq t_1,t_2\leq T,\qquad
    z_i\in B_{\mathbb R^{n_i}}(0,1),
\]
Cauchy--Schwarz in the parameter variables
shows that $F_m\to F$ in $L^1$.
Indeed, the parameter measure of this region
is finite, so its $L^1$ norm is bounded by
a fixed multiple of the $L^1(\mathcal H)$ norm.

On the other hand,
\eqref{eq:truncated-analysis-continuity}
and $\|f_m\|_1\to0$ imply that $F_m\to0$
in $L^2$ on this region.
Since its full measure is finite, this also
gives convergence to zero in $L^1$.
Uniqueness of the $L^1$ limit therefore implies
that $F=0$ there.

Exhausting the spatial variables by bounded
sets and taking, for example,
$\varepsilon=1/N$ and $T=N$, we obtain
$F=0$ almost everywhere.
Hence
\[
    \|S_{L_1,L_2}f_m\|_1
    =
    \|F_m\|_{L^1(\mathcal H)}
    \longrightarrow0.
\]
This proves injectivity.
\end{proof}

From now on, we identify $H^1_{L_1,L_2}$
with its image in $L^1$.

\subsection{Proof of the atomic characterization}
\label{subsec:proof-atomic-characterization}

We first record the passage from localized
Hardy norms to the atomic space.
This argument applies directly to $L^1$
functions and will also be used in the proof
of Theorem~\ref{thm:main-localization}.

\begin{lemma}\label{lem:localized-norms-to-atoms}
Suppose that $f\in L^1$ and
\[
    \sum_{Q_1\in\mathcal Q_1}
    \sum_{Q_2\in\mathcal Q_2}
        \|\phi_{Q_1}\phi_{Q_2}f\|
            _{h^1_{\ell(Q_1),\ell(Q_2)}}<\infty.
\]
Then $f$ admits an $L^1$-convergent representation
by atoms from Definition~\ref{def:covering-product-atom},
with
\begin{equation}\label{eq:localized-norms-to-atoms}
    \sum_\nu|\lambda_\nu|
    \leq
    C\sum_{Q_1\in\mathcal Q_1}
    \sum_{Q_2\in\mathcal Q_2}
        \|\phi_{Q_1}\phi_{Q_2}f\|
            _{h^1_{\ell(Q_1),\ell(Q_2)}}.
\end{equation}
In particular, the same upper bound holds
for $\|f\|_{H^1_{\mathrm{at}}(\mathcal Q_1,\mathcal Q_2)}$.
\end{lemma}

\begin{proof}
For each $Q_1,Q_2$, the function
$\phi_{Q_1}\phi_{Q_2}f$ belongs to
$h^1_{\ell(Q_1),\ell(Q_2)}$ and vanishes outside
$Q_1^*\times Q_2^*$.
Corollary~\ref{cor:atoms-inside-enlargements}
gives
\[
    \phi_{Q_1}\phi_{Q_2}f
    =
    \sum_\nu
        \lambda_{Q_1,Q_2,\nu}a_{Q_1,Q_2,\nu}
    \quad\text{in }L^1,
\]
where the atoms satisfy
Definition~\ref{def:covering-product-atom} and
\[
    \sum_\nu|\lambda_{Q_1,Q_2,\nu}|
    \leq
    C\|\phi_{Q_1}\phi_{Q_2}f\|
        _{h^1_{\ell(Q_1),\ell(Q_2)}}.
\]
Summing these bounds gives
\eqref{eq:localized-norms-to-atoms}.

By \eqref{eq:covering-product-atom-L1},
the combined atomic series converges absolutely
in $L^1$.
Moreover, \eqref{eq:partition-of-unity}
and \eqref{eq:product-partition-of-unity} give
\[
    \sum_{Q_1,Q_2}
        \|\phi_{Q_1}\phi_{Q_2}f\|_1
    =\|f\|_1,
    \qquad
    \sum_{Q_1,Q_2}\phi_{Q_1}\phi_{Q_2}f=f
    \quad\text{in }L^1.
\]
Grouping the atomic series according to
$Q_1,Q_2$ therefore shows that its sum is $f$.
\end{proof}

\begin{proof}[Proof of Theorem~\ref{thm:main-atomic-characterization}]
The identification with a subspace of $L^1$
has been established in
Lemma~\ref{lem:completion-L1-identification}.
We prove the atomic characterization in four steps.

\medskip
\noindent
\textbf{Step 1: A uniform Hardy norm bound for atoms.} Let $a$ be an atom associated with $Q_1,Q_2$.
By Definition~\ref{def:covering-product-atom},
it belongs to $L^2$, vanishes outside
$Q_1^{**}\times Q_2^{**}$, and is a local
product atom at scales
$\eta\ell(Q_1),\eta\ell(Q_2)$.
Applying Lemma~\ref{lem:product-supported},
\eqref{eq:local-scale-comparison-reverse},
and \eqref{eq:local-product-atom-bound}, we obtain
\begin{align}
    \|S_{L_1,L_2}a\|_1
    &\leq
    C\|a\|_{h^1_{\ell(Q_1),\ell(Q_2)}}
    \leq
    C\left(1+\log\frac1\eta\right)^2
    \|a\|_{h^1_{\eta\ell(Q_1),\eta\ell(Q_2)}}
    \leq C.
    \label{eq:covering-product-atom-square}
\end{align}
Since $\eta$ is fixed, the constant is
independent of the atom and its covering cubes.
In particular, every atom belongs to the
defining $L^2$ class of $H^1_{L_1,L_2}$.

\medskip
\noindent
\textbf{Step 2: Atomic decomposition on the defining $L^2$ class.} Let $f\in L^2$ with $S_{L_1,L_2}f\in L^1$.
By Corollary~\ref{cor:schrodinger-core-L1},
we have $f\in L^1$.
Furthermore,
\eqref{eq:product-localization-upper} gives
\begin{equation}\label{eq:atomic-localized-norm-sum}
    \sum_{Q_1\in\mathcal Q_1}
    \sum_{Q_2\in\mathcal Q_2}
        \|\phi_{Q_1}\phi_{Q_2}f\|
            _{h^1_{\ell(Q_1),\ell(Q_2)}}
    \leq C\|S_{L_1,L_2}f\|_1.
\end{equation}
Lemma~\ref{lem:localized-norms-to-atoms}
therefore provides an atomic representation
of $f$ in $L^1$ such that
$
    \sum_\nu|\lambda_\nu|
    \leq C\|S_{L_1,L_2}f\|_1
$.
Consequently,
\begin{equation}\label{eq:atomic-bound-on-L2-core}
    \|f\|_{H^1_{\mathrm{at}}(\mathcal Q_1,\mathcal Q_2)}
    \leq C\|S_{L_1,L_2}f\|_1.
\end{equation}

\medskip
\noindent
\textbf{Step 3: Convergence of summable atomic series.} Let
\[
    f=\sum_{\nu=1}^{\infty}\lambda_\nu a_\nu
    \quad\text{in }L^1,
    \qquad
    \sum_{\nu=1}^{\infty}|\lambda_\nu|<\infty.
\]
Every finite partial sum belongs to the
defining $L^2$ class.
The triangle inequality for the square function
and \eqref{eq:covering-product-atom-square} give
\[
    \left\|
        S_{L_1,L_2}
        \left(\sum_{\nu=M}^N\lambda_\nu a_\nu\right)
    \right\|_1
    \leq C\sum_{\nu=M}^N|\lambda_\nu|.
\]
Thus the partial sums converge in
$H^1_{L_1,L_2}$.
By \eqref{eq:schrodinger-completion-L1-embedding},
their Hardy space limit has the same $L^1$
representative as their $L^1$ limit, namely $f$.
Moreover,
$
    \|f\|_{H^1_{L_1,L_2}}
    \leq C\sum_{\nu=1}^{\infty}|\lambda_\nu|
$.
Taking the infimum over all atomic
representations yields
\begin{equation}\label{eq:atomic-to-schrodinger}
    \|f\|_{H^1_{L_1,L_2}}
    \leq
    C\|f\|_{H^1_{\mathrm{at}}(\mathcal Q_1,\mathcal Q_2)}.
\end{equation}

\medskip
\noindent
\textbf{Step 4: Atomic decomposition on the completion.} Let $f\in H^1_{L_1,L_2}$, $f\neq0$.
Choose functions $f_m$ in the defining
$L^2$ class such that
\[
    \|f_m-f\|_{H^1_{L_1,L_2}}
    \leq2^{-m}\|f\|_{H^1_{L_1,L_2}}.
\]
Then
\[
    f=f_1+\sum_{m=2}^{\infty}(f_m-f_{m-1})
\]
in $H^1_{L_1,L_2}$ and, by
\eqref{eq:schrodinger-completion-L1-embedding},
also in $L^1$.
The approximation bounds imply
\begin{equation}\label{eq:atomic-telescoping-bound}
    \|f_1\|_{H^1_{L_1,L_2}}
    +
    \sum_{m=2}^{\infty}
        \|f_m-f_{m-1}\|_{H^1_{L_1,L_2}}
    \leq3\|f\|_{H^1_{L_1,L_2}}.
\end{equation}
Indeed, the first term is at most
$\frac32\|f\|_{H^1_{L_1,L_2}}$, while the
$m$th difference is bounded by
$(2^{-m}+2^{-(m-1)})\|f\|_{H^1_{L_1,L_2}}$.

Each term in the telescoping series belongs
to the defining $L^2$ class.
Applying the construction in Step~2 to these
terms gives
\[
    f_1=\sum_\nu\lambda_{1,\nu}a_{1,\nu},
    \qquad
    f_m-f_{m-1}
    =
    \sum_\nu\lambda_{m,\nu}a_{m,\nu},
    \quad m\geq2,
\]
with
\[
    \sum_{m=1}^{\infty}\sum_\nu|\lambda_{m,\nu}|
    \leq C\|f\|_{H^1_{L_1,L_2}}
\]
by \eqref{eq:atomic-telescoping-bound}.

The combined series converges absolutely
in $L^1$ by \eqref{eq:covering-product-atom-L1},
and in $H^1_{L_1,L_2}$ by Step~3.
Its $L^1$ sum is $f$, since grouping the terms
according to $m$ recovers the telescoping series.
Lemma~\ref{lem:completion-L1-identification}
therefore identifies its Hardy space sum
with the same function $f$.
Enumerating its terms gives the required
atomic decomposition and
\[
    \|f\|_{H^1_{\mathrm{at}}(\mathcal Q_1,\mathcal Q_2)}
    \leq C\|f\|_{H^1_{L_1,L_2}}.
\]
Together with \eqref{eq:atomic-to-schrodinger},
this proves the theorem.
\end{proof}

\subsection{The localized norm characterization}
\label{subsec:localized-norm-characterization}

We now extend
Theorem~\ref{thm:product-localization-comparison}
from the defining $L^2$ class to all
$L^1$ functions.

\begin{proof}[Proof of Theorem~\ref{thm:main-localization}]
\medskip
\noindent
\textbf{Step 1: Localization of a Hardy space function.}

Let $f\in H^1_{L_1,L_2}$ and choose $f_m$
in the defining $L^2$ class such that
\[
    \|f_m-f\|_{H^1_{L_1,L_2}}\longrightarrow0.
\]
By \eqref{eq:schrodinger-completion-L1-embedding},
\[
    \|f_m-f\|_1
    \leq C\|f_m-f\|_{H^1_{L_1,L_2}}
    \longrightarrow0.
\]
For each $Q_1,Q_2$, boundedness of the
partition functions therefore gives
\[
    \phi_{Q_1}\phi_{Q_2}f_m
    \longrightarrow
    \phi_{Q_1}\phi_{Q_2}f
    \quad\text{in }L^1.
\]
Lemma~\ref{lem:local-Hardy-Fatou} implies
\[
    \|\phi_{Q_1}\phi_{Q_2}f\|
        _{h^1_{\ell(Q_1),\ell(Q_2)}}
    \leq
    \liminf_{m\to\infty}
    \|\phi_{Q_1}\phi_{Q_2}f_m\|
        _{h^1_{\ell(Q_1),\ell(Q_2)}}.
\]
Applying Fatou's lemma to the sum over the
covering cubes and using
\eqref{eq:product-localization-upper}, we obtain
\begin{align}
    &\sum_{Q_1\in\mathcal Q_1}
    \sum_{Q_2\in\mathcal Q_2}
        \|\phi_{Q_1}\phi_{Q_2}f\|
            _{h^1_{\ell(Q_1),\ell(Q_2)}}
    \leq
    \liminf_{m\to\infty}
    \sum_{Q_1,Q_2}
        \|\phi_{Q_1}\phi_{Q_2}f_m\|
            _{h^1_{\ell(Q_1),\ell(Q_2)}}
    \notag\\
    &\qquad\leq
    C\lim_{m\to\infty}\|f_m\|_{H^1_{L_1,L_2}}
    =
    C\|f\|_{H^1_{L_1,L_2}}.
    \label{eq:full-localization-upper}
\end{align}

\medskip
\noindent
\textbf{Step 2: Reconstruction from localized Hardy norms.} Conversely, let $f\in L^1$ and suppose that
the sum of localized norms is finite.
Lemma~\ref{lem:localized-norms-to-atoms}
gives $f\in H^1_{\mathrm{at}}(\mathcal Q_1,\mathcal Q_2)$.
By \eqref{eq:atomic-to-schrodinger}
and \eqref{eq:localized-norms-to-atoms},
\[
    \|f\|_{H^1_{L_1,L_2}}
    \leq
    C\|f\|_{H^1_{\mathrm{at}}(\mathcal Q_1,\mathcal Q_2)}
    \leq
    C\sum_{Q_1\in\mathcal Q_1}
    \sum_{Q_2\in\mathcal Q_2}
        \|\phi_{Q_1}\phi_{Q_2}f\|
            _{h^1_{\ell(Q_1),\ell(Q_2)}}.
\]
Together with \eqref{eq:full-localization-upper},
this proves the equivalence of membership
and the asserted norm comparison.
\end{proof}

\section*{Declarations}

\subsection*{Use of Generative AI and AI-Assisted Technologies}
During the preparation of this manuscript, the authors utilized Google's Gemini (Thinking 3.6) and OpenAI's GPT (GPT-6-Astra) to assist with initial proof strategy brainstorming, language refinement, typesetting equations, and drafting preliminary standard technical lemmas. 

The author reviewed and edited all content as necessary and takes full responsibility for the accuracy, mathematical correctness, and overall integrity of the final published work.

\subsection*{Competing Interests}
The author declare no competing financial or non-financial interests directly relevant to the content of this article.

\subsection*{Data Availability}
Data sharing is not applicable to this article as no datasets were generated or analyzed during the current study.

\end{document}